\documentclass[11pt]{amsart}
\usepackage{pst-tree}
\usepackage{mathrsfs}
\usepackage{amsthm, amssymb, amsmath, pst-tree} 
\usepackage{amscd, amsfonts, verbatim,subfigure}
\usepackage[all]{xy}

\newcommand{\iso}{\cong}
\newcommand{\Z}{\mbb Z}

\newcommand{\tr}{\Delta}

\newcommand{\mbb}{\mathbb}
\newcommand{\mc}{\mathcal}
\newcommand{\A}{\mc A}
\newcommand{\R}{\mbb R}
\newcommand{\C}{\mbb C}
\newcommand{\abs}[1]{\left| #1  \right|}
\renewcommand{\d}{\mathrm{d}}

\newcommand{\op}{\operatorname}
\newcommand{\Met}{\op{Met}}

\newcommand{\cinfty}{C^\infty}
\renewcommand{\Tr}{\op{Tr}}
\newcommand{\br}{\overline}

\newcommand{\limdir}{\varinjlim}
\renewcommand{\det}{\op{det}}

\newcommand{\End}{\operatorname{End}}

\newtheorem{thm}{Theorem}
\newtheorem{cor}[thm]{Corollary}
\newtheorem{lem}[thm]{Lemma}
\newtheorem{prop}[thm]{Proposition} 
\theoremstyle{definition}
\newtheorem{defn}[thm]{Definition}
\newtheorem{remk}[thm]{Remark}
\newtheorem{ex}[thm]{Example}
\newtheorem{definition}[thm]{Definition}
\newtheorem{theorem}[thm]{Theorem}

\newtheorem{lemma}[thm]{Lemma}
 \newtheorem*{ack}{Acknowledgements}

\DeclareMathOperator{\dens}{Densities}
\DeclareMathOperator{\Hom}{Hom}

\begin{document}

\title[Closed String TCFT for Hermitian Calabi-Yau Elliptic Spaces]
{Closed String TCFT for Hermitian Calabi-Yau Elliptic Spaces}

\author[K.~Costello]{Kevin~Costello}
\address{Kevin Costello, Department of Mathematics, Northwestern University,
2033 Sheridan Road Evanston, IL 60208-2730, USA}
\email{costello@math.northwestern.edu}

\author[T.~Tradler]{Thomas~Tradler}
\address{Thomas Tradler, Department of Mathematics, College of Technology of the City University
of New York, 300 Jay Street, Brooklyn, NY 11201, USA}
\email{ttradler@citytech.cuny.edu}

\author[M.~Zeinalian]{Mahmoud~Zeinalian}
\address{Mahmoud Zeinalian, Department of Mathematics, C.W. Post Campus
of Long Island University, 720 Northern Boulevard, Brookville, NY
11548, USA} \email{mzeinalian@liu.edu}

\begin{abstract}
We describe an explicit action of the prop of the chains on the moduli space of Riemann surfaces on the Hochschild complex of a  Calabi-Yau elliptic space.   One example of such an elliptic space extends the known string topology operations, for all compact simply-connected manifolds, to a collection indexed by the de Rham currents on the moduli space. Another example pertains to the B-model at all genera.
\end{abstract}

\maketitle \setcounter{tocdepth}{1} \tableofcontents

\section{Introduction}

In \cite{C3}, the first author gave a construction of an open topological conformal field theory (TCFT) for a Calabi-Yau elliptic space using the machineries of heat kernels and differential forms on a certain moduli space of metrised ribbon graphs as a piecewise linear space. The most basic example of a Calabi-Yau elliptic space is the de Rham complex of differential forms on a closed and oriented manifold. Other examples of Calabi-Yau elliptic spaces come from holomorphic vector bundles on a Calabi-Yau manifold,  Yang-Mills bundles on a $4$-manifold \cite{C3},   or certain vector bundles on G2 manifolds \cite{DT}. 

The aim of this paper is to extend the above open TCFT to a closed TCFT, in an explicit way. 

There are two existing constructions of a closed TCFT from a cyclic $A_\infty$ algebra, due to Kontsevich and Soibelman \cite{K, KS2} and the first author \cite{C1}.   One can turn a Calabi-Yau elliptic space into a finite dimensional cyclic $A_\infty$ algebra, using the explicit form of the homological perturbation lemma \cite{KS1}.    Thus, both of the existing constructions apply in our situation.

In addition, the results of the first author \cite{C1} give an abstract way to construct a closed TCFT from any open TCFT; one can apply this directly to the open TCFT constructed in \cite{C3}.

One could thus ask why there is a need for a new construction.  The answer is that the existing constructions are ungeometric.      For instance, in the case of the de Rham complex of a manifold $M$, the existing constructions start by replacing the de Rham complex by the cohomology $H^\ast(M)$ and, from there, using homological algebra.  By contract, the construction we give in this paper works directly with the geometry of the manifold $M$.

The results of \cite{C1} (where the closed TCFT was constructed to satisfy a universal property) imply that the geometric constructions of this paper are equivalent to the algebraic constructions of \cite{C1}.

Thus, in this paper,   we will define an explicit action of a model for the prop of chains on the moduli space of Riemann surfaces, on the Hochschild complex of a Calabi-Yau elliptic space.   The construction depends on a Hermitian metric on the  CY elliptic space (in the example of the de Rham complex of a manifold, this comes from a metric on the manifold).  However, the resulting structure is independent up to homotopy of the choice of the Hermitian metric.

To this end, we will define the prop of meta-graphs $MG$, which we will show in Proposition \ref{MG=S} to be weakly equivalent to Segal's prop $\mc S$.     We will show that the prop $MG$ has an explicit lax action on a certain model for the Hochschild chain complex of $\A$. That is:
\setcounter{thm}{30} 
\begin{thm}
There is a lax algebra over the prop $C_\ast(MG)$ which sends
$$
n \mapsto L_\ast(\A)[n]
$$
where $L_\ast(\A)[n]$ is a certain chain complex quasi-isomorphic to the $n^{th}$ tensor power of the Hochschild chain complex of $\A$.

The structure of lax algebra is given by chain maps
$$
C_\ast(MG(n,m) ) \otimes L_\ast(\A)[n] \to L_\ast(\A)[m]
$$
compatible with differentials, composition, and tensor product maps.
\end{thm}
\setcounter{thm}{0}

This in turn, after passing to homology, gives a host of operations on the Hochschild homology of a Calabi-Yau elliptic space, indexed by the homology of the Segal prop. That is to say,

\setcounter{thm}{31} 
\begin{cor}
After passing to homology the above lax algebra structure yields a homological field theory on the Hochschild homology. That is, there is the following map of props,
$$H_\ast(\mc S(n, m)) \to \text{Hom} (HH_\ast (\mc A, \mc A)^{\otimes n}, HH_\ast (\mc A, \mc A)^{\otimes m} ).$$
\end{cor}
\setcounter{thm}{0}

The most basic example of a Calabi-Yau elliptic space is the de Rham complex of a closed and oriented manifold $M$. In this case, if the manifold at hand is simply connected, one gets operations on the cohomology of the free loop space indexed by the homology of the moduli space. 

Of course, as we remarked earlier, the results of \cite{KS2} and \cite{C1}, applied to the cyclic A-infinity algebra structure on the cohomology of $M$,  give all-genus constructions of string topology operations.  The construction here is equivalent, but much more geometric.

The geometric nature of the construction of this paper allows direct comparison with the existing string topology operations. In Section \ref{string-topology} we will show that our operations are are genuine extensions of the existing  operations to the homology of the entire moduli space of Riemann surfaces.

\setcounter{thm}{35} 
\begin{thm}
For a ribbon graph $\gamma\in G(n,m)$, let $\tilde{\gamma}$ denote the graph obtained by collapsing all the chord pieces to a point $\tilde{\gamma}=\gamma/\sim$. We obtain an operation $[\gamma]_\ast$ from Theorem \ref{main-thm}, and there is an induced string topology operation coming from $\tilde{\gamma}$. Then the following diagram commutes up to sign, where the vertical maps are isomorphisms:
\begin{equation*}
\xymatrix{  H_\ast(L_\ast(\Omega^\ast M)[n])
\ar[rr]^{[\gamma ]_\ast} \ar[d] && H_\ast(L_\ast(\Omega^\ast M)[m])  \ar[d] \\
 H^\ast(\mc L M)^{\otimes n}  \ar[rr]^{\text{operation from } \tilde{\gamma}} && H^\ast(
\mc L M)^{\otimes m} }
\end{equation*}
\end{thm}
\setcounter{thm}{0}

In fact, adopting the language of lax algebras, Theorem \ref{main-thm} may be viewed as giving a chain level description of these string topology operations.  Recently, V. Godin \cite{Go} has given a independent construction of string topology operations at the homology level. 

Another example of a Calabi-Yau elliptic space arises from a vector bundle on a Calabi-Yau variety.  If the vector bundle is a strong generator of the derived category of the Calabi-Yau, then we conjecture that the closed TCFT which we construct explicitly gives the B model at all genera.  This material is explained in  section \ref{B-model}.

Here is a brief description of each chapter.  Chapter 2 recalls the definition of a  Calabi-Yau elliptic space. Such objects were originally considered by Kontsevich several years ago, and used by the first author in \cite{C3}.  These are particular differential graded algebras associated to a manifold with some extra structure.   We are particularly interested in two examples of such objects: the de Rham complex of a smooth manifold, and the elliptic space associated to a Calabi-Yau manifold with a vector bundle.  (At the moment, it is not clear to us what the meaning of the TCFT's constructed from Yang-Mills bundles and from G2 manifolds should be). 

Section 3 describes two moduli spaces of Ribbon graphs $G(n, m)$ and $\Gamma(n, m)$. Each of these moduli spaces is homotopy equivalent to certain moduli spaces of Riemann surfaces.  The moduli space $\Gamma(i, j)$ was used in \cite{C3} in order to construct an open TCFT.  The goal of this paper is to construct the corresponding closed version. The relationship between this closed theory to the open version should be viewed in the light of the results of \cite{C1}.

One of the problems in dealing with the family of moduli spaces $G(n, m)$ is that they do not , unlike their open counterparts $\Gamma(i, j)$, form a prop. In Section 4 we show how to use meta-graphs to overcome this problem. A meta-graph is a directed graph each of whose internal vertices is decorated by an appropriate element of $G(n, m)$, for suitable numbers $n$ and $m$.

In section 6 we talk about the notion of de Rham differential forms and de Rham currents on piecewise linear spaces. The notion of such differential forms on such spaces goes back to Whitney. Spaces $G(n, m)$, $\Gamma(i, j)$, and $\Delta_n$ are examples of piecewise linear spaces used in this paper.

In section 7 we modify the definition of the Hochschild complex in order to obtain a more suitable model for our construction.

Section 8 is the heart of the paper. It is here that the action is constructed. This action is in spirit similar to the action on the Hochschild complex of an associative algebra with nondegenerate, invariant inner product constructed in \cite{TZ}. Also, see Kaufmann \cite{Ka1, Ka2} for a version of an action on the Hochschild complex.

Section 9 is devoted to the example of the de Rham complex.  We show that after passing to the homology, our operations labelled by the zeroth classes, on the cohomology of the free loop space, coincide with the previously defined string topology operations. We do so in analogy to an argument identifying the string topology product with the Gerstenhaber cup product that was given by Felix and Thomas in \cite{FT}. 

Section 10 speculates on a relationship between the operations of this paper on the Hochschild complex of the Dolbeault resolution of a endomorphisms of a perfect complex and the B-model at all genera, when the complex in question is the the strong generator in the category of perfect complexes on an algebraic Calabi-Yau elliptic space.

\begin{ack}
We would like to thank Dennis Sullivan and Maxim Kontsevich for helpful conversations on this topic. We also would like to thank Gr\'egory Ginot for his useful comments on identifying the string topology operations. The second and third authors thank the Max Planck Institute for the support during their visits.
\end{ack}

\section{Calabi-Yau elliptic space}

Let us recall some basic definitions about Calabi-Yau elliptic spaces from \cite{C3}.

\begin{defn}
An elliptic space $(M,\mathcal A)$ consists of a smooth compact
manifold $M$ and a bundle of finite dimensional $\mathbb
Z_2$-graded associative algebras over $\mathbb C$ on $M$ whose
$\mathbb Z_2$-graded algebra of sections is denoted by $\mathcal
A$. Moreover, $\mathcal A$ is endowed with a first order
differential operator $Q : \mathcal A \to \mathcal A$, which is an
odd derivation of square zero making $\mathcal A$ into an elliptic
complex.
\end{defn}

\begin{defn}
Let $(M, \mathcal A)$ be an elliptic space. A Calabi-Yau structure,
CY for short, is a $\mathbb C$-linear trace map $\Tr : \mathcal A
\to \mathbb C$, satisfying $\Tr(ab)= \Tr(ba)$ and $\Tr(Qa)=0$.  It
assumed that the trace map factors through the densities, that is to
say $\Tr: \mathcal A \to \text{Densities(M)} \to \mathbb C$, where
the first arrow is $\cinfty(M,\C)$ linear. Here $\text{Densities(M)}$ is the vector space concentrated in degree zero of all sections of the top exterior algebra of the cotangent bundle twisted with the orientation bundle. The parity of this map
to the densities is denoted by $p(\mathcal A)$.
The pairing $\Tr(ab)$ on $\A$ is assumed to be non-degenerate, in
the sense that the induced map $\A \to \Hom_{\cinfty(M,\C)}(\A,
\dens(M))$ is an isomorphism.
\end{defn}

\begin{defn}
A Hermitian metric $\langle \cdot,\cdot \rangle: \mathcal A
\otimes \mathcal A \to \mathbb C$ on $\mathcal A$ is said to be
compatible with a CY elliptic space $(M, \mathcal A)$ if there
exists a complex antilinear, $C^\infty (M, \mathbb R)$-linear
operator $\ast: \mathcal A \to \mathcal A$,
satisfying $\ast\ast a =(-1)^{|a|( p(A)+1)}$, (where $|a|$
denotes the parity of $a \in \mathcal A$) such that 
$$\langle a, b\rangle=\Tr(a \cdot \ast
b).$$
\end{defn}
Let $Q^\dag$ be the Hermitian adjoint of $Q$, and let $H =
[Q,Q^\dag]$.  $H$ is an elliptic operator on $\A$, self adjoint with
respect to the Hermitian metric.

If $E,F$ are vector bundles on manifolds $M,N$, and $\mc E, \mc F$
are their spaces of global sections, we will use the notation $\mc E
\otimes \mc F$ to refer to the space of global sections of $E
\boxtimes F$ on $M \times N$.  In other words, $\otimes$ will refer
to the completed projective tensor product where appropriate.

The heat kernel $K(x, y, t) \in \mathcal A \otimes \mathcal A
\otimes \mathbb C^\infty( \R_{> 0} ) $ for a Hermitian CY elliptic space $(M,
\mathcal A)$ satisfies
$$e^{-tH} \alpha (x) = (-1)^{p(\A) \abs \alpha } \op{\Tr}_{y} K(x, y,
t)\alpha(y)$$ where $H= [Q, Q^\dagger]$.   Since $H$ is a
self-adjoint elliptic operator on a compact manifold, it admits a
heat kernel; see Gilkey \cite{G}.

\begin{ex} Let $M$ denote a compact and oriented Riemannian
manifold. It should be easy to see that $Q= \d$ makes $\mathcal A =
\Omega^\ast (M, \mathbb C)$ into an elliptic space. Moreover,
$\Tr(a) =\int_M a$ makes this into a CY elliptic space.  A metric on
$M$ induces a Hermitian metric on $\A$, such that $Q^\dag = \d^\ast$
and $H$ is the Laplacian on forms.
\end{ex}

\begin{ex}
Let $M$ denote a compact Calabi-Yau manifold of dimension $n$,
with a holomorphic volume form $Vol_M$. Let $E$ be a complex of
holomorphic vector bundles on $M$.  Let $\mathcal A= \Omega^{0, \ast}(M)
\otimes_{C^\infty(M)} \End(E)$, and $Q=\overline
\partial$.  Define $\Tr (\alpha)= \int_M \text{tr}(\alpha)\wedge Vol_M$.  

Choice of a Hermitian metric on $M$ and a Hermitian metric on the bundle $E$ leads to a Hermitian metric on the Calabi-Yau elliptic space $\mc A$. 
\end{ex}
Other examples are discussed in \cite{C3,DT}.

\section{Ribbon graphs}

In this section, we define the ribbon graph spaces $\Gamma(n,m)$ and $G(n,m)$ which are essential for our purposes. We furthermore cite a theorem from \cite{C3} which makes a given CY elliptic space $(M,\mc A)$ into an algebra over the prop $\Gamma(n,m)$, see Theorem \ref{Kevins-map}. For the sake of completeness, let us start by recalling the definition of a prop.

\begin{definition}
A prop is a (not necessarily unital) symmetric monoidal category whose monoid of objects is $\Z_{\ge 0}$.

A topological prop is such a symmetric monoidal category enriched over topological spaces. Thus, the morphisms $P(n,m)$ in a topological prop $P$ are topological spaces, and the structure maps are continuous maps.

A differential graded prop is such a symmetric monoidal category enriched over chain complexes. Thus, the morphisms $P(n,m)$ in a differential graded prop are chain complexes, and all structure maps are multi-linear maps compatible with the differentials.

If $\mc C, \mc D$ are dg symmetric monoidal categories, a lax symmetric monoidal functor $\mc C \to \mc D$ is a functor $F: \mc C \to \mc D$, with tensor product maps
$$
F(c_1) \otimes F(c_2) \to F(c_1 \otimes c_2)
$$
which are ``associative'' and ``commutative''.

A lax algebra over a dg prop $P$ is a dg symmetrical monoidal functor $F$ from $P$ to the dg category of chain complexes, such that the tensor product maps
$$
F( n ) \otimes F(m) \to F(n + m)
$$
are quasi-isomorphisms.
\end{definition}

Let us now recall the definition of the spaces $\Gamma(n,m)$ from \cite{C3}.

\begin{defn}
A graph is a compact one dimensional cell complex. A ribbon graph is
a graph with a cyclic order on the set of germs of edges emanating
from each vertex.

Let $\Gamma(n,m)$ be the space of metrised ribbon graphs with $n$
incoming and $m$ outgoing labeled external vertices. These external
vertices are uni-valent; all other vertices are at least tri-valent.
An edge adjoining an external vertex is called an external edge.

These graphs need not be connected.  We require that metrised ribbon
graphs have no loops of length zero, and no path of length zero
connecting two outgoing vertices. We allow the graph which has only a single edge, connecting two external vertices, as long as the edge is of non-zero length if the two vertices are outgoing.   Also we allow graphs with
connected components of this form.

The spaces $\Gamma(n,m)$ form a topological prop, or equivalently a
topological symmetric monoidal category whose objects are the
non-negative integers $\Z_{\ge 0}$.
\end{defn}

In section \ref{strat}, we give some definitions about spaces with piecewise-linear stratifications.   $\Gamma(n,m)$ has such a stratification, in a natural way; two graphs are in the same stratum if they are homeomorphic in a way respecting all the markings, and the ribbon structure, but not
necessarily the metric.    To any space $X$ with a PL stratification we define in section \ref{strat} a complex $\Omega^\ast(X)$ of piecewise smooth differential forms.  Let $\det$ be the local system on the spaces $\Gamma(n,m)$, defined in \cite{C3}. The fibre of $\det$ above a graph $\gamma \in \Gamma(n,m)$ is the determinant of the cohomology of $\gamma$
relative to the $m$ outgoing external vertices, with a shift in
grading determined by the Euler characteristic of these cohomology
groups.

\begin{thm}[\cite{C3}, section 4]\label{Kevins-map}
For every CY elliptic space $(M,\A)$, $\A$ is an algebra over the
prop $\Gamma$, in the following sense.  That is, there are maps
$$ K: \A^{\otimes n}\to \Omega^\ast(\Gamma(n,m), \det^{
p(\A)}) \otimes \A^{\otimes m}. $$ Gluing of graphs corresponds to
composition of morhpisms $\A^{\otimes n} \to \A^{\otimes m}$, and
disjoint union of graphs corresponds to tensor products of these
morphisms.
\end{thm}

The main goal of this paper is to develop a closed string version of
the above action. To this end, we will introduce an appropriate
model for a prop of moduli space of Riemann surfaces with closed
boundaries.

\begin{defn} Let $G(n,m)$ be the moduli space of metrised ribbon graphs
$\gamma$, with $n+m$ numbered boundary components, each of which is
equipped with precisely one external edge.  These external edges are
of length zero.  The first $n$ boundary components are
incoming, the last $m$ outgoing. The outgoing boundary components
are disjoint embedded circles. Here, ``embedded circle" means
non-self intersecting; there is no edge in $\gamma$ both sides of
which are on an outgoing boundary component.

We require the following two conditions on the metric of a graph $\gamma$.
Frist, there are no loops of length zero. Next, after removing the output circles from the graph $\gamma$, we obtain a new graph, whose connected components we call the chords of $\gamma$. We then require, that the distance between any two endpoints of a chord has to have a positive length.
\end{defn}

Note that in particular this means that each connected component of
a graph in $G(n,m)$ has at least one incoming boundary component.

\section{Enlarging $G(n,m)$ to a prop $MG$}

The spaces $G(n,m)$ don't quite form a prop, as the compositions are only partially defined.  If $\Sigma_1 \in G(n,m)$, $\Sigma_2 \in G(m,r)$, then we can only glue $\Sigma_1$ to $\Sigma_2$ when all the outgoing boundary components of $\Sigma_1$ are of the same length as the incoming boundary components of $\Sigma_2$.  We will replace this partially-defined prop $G$ by a ``rectified'' version, which is an actual prop.  This rectified version will be called $MG$.

We will use throughout an explicit description of $MG$.  However, it is possible to give an abstract definition, as follows.
The prop $MG$ can be defined by saying that it is the universal prop with a map $\Phi : G(n,m) \to MG(n,m)$, such that
\begin{enumerate}
\item
if  $\Sigma_1 \in G(n,m)$, $\Sigma_2 \in G(m,r)$ are composable, $\Phi ( \Sigma_1 \circ \Sigma_2 ) = \Phi ( \Sigma_1 ) \circ \Phi( \Sigma_2 )$.
\item
if $\Sigma_i \in G(n_i,m_i)$, then $\Phi( \Sigma_1 \amalg \Sigma_2 ) = \Phi ( \Sigma_1 ) \amalg \Phi ( \Sigma_2 ) $.
 \end{enumerate}

Now we will give the concrete description.
\begin{defn}
A directed graph is a graph, with an orientation on each edge, such that
\begin{enumerate}
\item
Every vertex has at least one incoming edge.
\item
There are no oriented loops, that is, there is no loop in the graph
which always follows the orientation of the edges.
\end{enumerate}
A vertex $v$ in a directed graph $\gamma$ has level $l(v)=k$ if the
longest path, starting with an incoming edge and ending at $v$, is
of length $k$.

The condition that the graph contains no oriented loops implies that
every vertex has a level $l(v)$ with $1 \le l(v) < \infty$.
\end{defn}

\begin{defn}
A meta-graph is given by
\begin{enumerate}
\item
A connected, directed graph $\gamma$, as above.
\item
At each vertex $v$ of $\gamma$, an element of
$$RG_v \in  G(v_{in}, v_{out})$$ where $v_{in}, v_{out}$ refer to the set of incoming and outgoing edges at $v$.
\end{enumerate}
\end{defn}
Let $e \in E(\gamma)$ be an edge from $v_1$ to $v_2$.  The edge $e$
corresponds to an outgoing boundary component of the ribbon graph
$RG_{v_1}$, and an incoming boundary component of $RG_{v_2}$. Define
the \emph{discrepancy} along $e$ to be the difference in lengths
between these two boundary components.

\begin{defn}
Let $MG(n,m)$ denote the space of meta-graphs $\gamma$, with $n$
incoming and $m$ outgoing edges, modulo the following equivalence
relation.

Let $v_1,v_2$ be two edges of $\gamma$, and suppose that the longest
directed path of edges, starting at $v_1$ and end at $v_2$, is of
length one.   Let $e_1,\ldots,e_k$ be the edges of $\gamma$ which
start at $v_1$ and ending at $v_2$.    Suppose that each of the
$e_i$ has zero discrepancy.  Then, we can form a new meta-graph
$\gamma/ (e_1, \ldots, e_k ) $, by gluing $RG_{v_1}$ to $RG_{v_2}$.
In this situation, we identify
$$
\gamma = \gamma / (e_1,\ldots,e_k ) .
$$
\end{defn}
\begin{remk}
Note that we  only perform this identification when \emph{all} edges
connecting $v_1$ to $v_2$ have zero discrepancy.  Otherwise, the
result would not be a meta-graph; there would be oriented loops.
Similarly, we can only perform this identification when the longest
directed path from $v_1$ to $v_2$ is of length one, as otherwise we
would introduce oriented loops.
\end{remk}

Note that there is a natural inclusion $G(n,m) \to MG(n,m)$ which maps to the the trivial meta-graph.
\begin{prop}
$MG(n,m)$ deformation retracts onto $G(n,m)$.
\label{theorem deformation retract}
\end{prop}
\begin{proof}
We will define a map
$$\Phi : MG(n,m) \times [0,\infty] \to MG(n,m)$$ such that $\Phi(\gamma,\infty) \in G(n,m)$ for
all $\gamma \in MG(n,m)$,  $\Phi(\gamma,0) = \gamma$ for all $\gamma \in MG(n,m)$, and  $\Phi(\gamma,t)=  \gamma$ if $\gamma
\in G(n,m)$, for all $t \in [0,\infty]$.

Let $\gamma \in MG(n,m)$.  Suppose that the largest level of any vertex $v \in V(\gamma)$ is $k$.

Let $e$ be an edge of $\gamma \in MG(n,m)$, from $v_1$ to $v_2$,
such that $v_2$ is of level $k$.   Let $B_1, B_2$ denote  the
boundary components of $RG_{v_1}, RG_{v_2}$ corresponding to $e$,
and let $l_1,l_2$ denote the lengths of $B_1,B_2$.

Recall that $B_1$ is an embedded circle on $RG_{v_1}$, with no self
intersections.  Thus, we can get a new metrised ribbon graph by
uniformly scaling the length of $B_1$ to any desired positive
length, while leaving the lengths of all other edges of $RG_{v_1}$
unchanged.

If $l_1 \le l_2$, then in $\Phi(\gamma,t)$ we change the length of $B_1$ to
$$
l_1(t) = \op{min} ( e^t l_1 , l_2 ) .
$$
If $l_1 \ge l_2$,  we change the length of $B_1$ to
$$
l_1(t) = \op{max} ( e^{-t} l_1 , l_2 ) .
$$

Repeat this procedure for all edges $e$ of $\gamma$ which end at a vertex of level $k$.

After we do this, all ribbon graphs $RG_v$ attached to vertices of level $\le k-1$ are $t$-dependent.

Next, let $e$ be an edge which ends on a vertex $v_2$ of level
$k-1$, and starts at a vertex $v_1$.   As before, we have boundary
components $B_1,B_2$ of $RG_{v_1}, RG_{v_2}$ respectively.    Recall
$RG_{v_2}$ has some $t$-dependence.  Thus, the length $l_2(t)$ of
$B_2$ is $t$-dependent.  The length $l_1$ of $B_1$ is independent of
$t$, because the embedded boundary $B_1$ in $RG_{v_1}$ is disjoint
from any boundary circle whose length we have already changed.

Now let's repeat the procedure as above, but using the $t$-dependent
length $l_2(t)$.  That is, If $l_1 \le l_2(t)$, then we change the
length of $B_1$ to
$$
l_1(t) = \op{min} ( e^t l_1 , l_2 (t)) .
$$
If $l_1 \ge l_2(t)$,  we change the length of $B_1$ to
$$
l_1(t) = \op{max} ( e^{-t} l_1 , l_2 (t) ) .
$$
Let's repeat this for all edges $e$ which end on vertices of level
$k-1$, and then for all edges which end on vertices of level $k-2$,
and so on.

We need to check that this procedure is well-defined.  That is, let
$\gamma_1,\gamma_2$ be meta-graphs which are equivalent under the
equivalence relation defining $MG(n,m)$.  We need to check that
$\Phi(\gamma_1,t)$ and $\Phi(\gamma_2,t)$ are equivalent.

Suppose that $\gamma_2$ is obtained from $\gamma_1$ by contracting
all the edges going from a vertex $v_1$ to a vertex $v_2$, where the
discrepancy along each edge is zero.  Let $e$ be such an edge.  Note
that the level of $v_2$ is greater than that of $v_1$.

As before, let $B_1, B_2$ be the boundary components of $RG_{v_1},
RG_{v_2}$ corresponding to $e$, and let $l_i(t)$ be their
$t$-dependent lengths,   At $t = 0$, $l_1(0) = l_2(0)$, because
there is no discrepancy along $e$.  We need to check that $l_1(t) =
l_2(t)$ for all $t$.  This will imply that $\Phi(\gamma_1,t)$ and $\Phi(\gamma_2,t)$ are equivalent.

Recall that every outgoing boundary circle of $RG_{v_2}$ has been
rescaled up or down by at most $e^t$. The edges of $RG_{v_2}$ which
are not on a boundary circle have not been rescaled.  Thus, every
edge of $RG_{v_2}$ has been rescaled up or down by at most $e^t$; it
follows that
$$e^{-t} l_2(0) \le l_2(t) \le e^t l_2(0).$$

Since $l_2(0) = l_1(0)$, it follows that $e^{-t} l_1(0) \le l_2(t)
\le e^t l_1(0)$.  By definition, if $l_1(0) \ge l_2(t)$ then $l_1(t)
= \op{max}(e^{-t} l_1(0), l_2(t) ) = l_2(t)$, and similarly if
$l_1(0) \le l_2(t)$.  Thus, $l_1(t) = l_2(t)$ for all $t$.

Thus, our map $\Phi : MG(n,m) \times [0,\infty] \to MG(n,m)$ is
well-defined and continuous.  At value $\infty$ it lands in
$G(n,m)$, and at $0$ it is the identity.  Further, if $\gamma \in
G(n,m)$, then $\Phi(\gamma,t)  = \gamma$ for all $t \in [0,\infty]$.
Thus, $\Phi$ is a deformation retraction, as required.
\end{proof}

It is easy to see that the spaces $MG(n,m)$ form a prop.  The maps
$MG(n,m) \times MG(m,r) \to MG(n,r)$ are given by gluing directed
graphs to each other in the evident way.    This map respects the
equivalence relation defining the spaces $MG(n,m)$.  Let $MG$ denote
this prop.

\section{Moduli spaces of Riemann surfaces}

We now show, that the prop $MG$ is a model for Segal's prop $\mc S$ of moduli spaces.
\begin{defn}
Let $\mathcal S(n, m)$ denote the infinite dimensional space of all
isomorphism classes of Riemann surfaces with $n$ incoming and $m$ outgoing boundary components.  These boundary components are given an analytic parametrisation.  We require that every connected component of each surface in $\mc S(n,m)$ has at least one
incoming boundary. The composition of morphisms in
this prop is given by sewing the Riemann surfaces along the
boundaries (this is well-defined because the parametrisations are analytic). The tensor product of morphisms is the disjoint union.
\end{defn}

\begin{prop}\label{MG=S}
The prop $MG$ is weakly equivalent to the prop $\mc S$.
\end{prop}
\begin{proof}[Proof (sketch)]

First, we will define a sub-prop of $MG$ which is weakly equivalent to $MG$.  Let $\gamma \in MG(n,m)$ be a meta-graph. Each of the $n$ incoming edges of $\gamma$ is attached to a vertex, and corresponds to an incoming boundary component on that vertex. Let $l^{in}_1,\ldots,l^{in}_n$ denote the lengths of these incoming boundary components. In a similar way,let $l^{out}_1,\ldots,l^{out}_m$ denote the lengths of the outgoing boundary components.  We will define $MG'(n,m)$ to be the subspace where $l_i^{out} = 1$, for all $i$, and $l_j^{in} = 2$, for all $j$.  It is clear that $MG'(n,m)$ is a sub-prop, and it is not difficult to see that the $MG(n,m)$ deformation retracts onto $MG'(n,m)$.

Now, we will define a map $MG' (n,m) \to \mc S(n,m)$ which is a weak equivalence of props.

Let $\gamma \in MG'(n,m)$.  Let $e$ be an interior edge of $\gamma$, from $v_1$ to $v_2$.  Recall each vertex $v$ of $\gamma$ is labeled by a ribbon graph $RG_{v} \in G(v_{in}, v_{out} )$.  Let $l_1$ denote the length of the outgoing boundary component of $RG_{v_1}$ corresponding to $e$, and let $l_2$ denote the length of the incoming boundary component of $RG_{v_2}$ corresponding to $e$.

 We will construct a surface $\Sigma(\gamma)$ as follows.
\begin{enumerate}
\item
Replace each edge interior $e$ of $\gamma$ by the annulus
$$
\{z \in \C \mid l_1 \le \abs{z}\le l_2 \}
$$
with the flat metric induced from $\C$.
\item
Replace each of the $n$ incoming edges of $\gamma$ by the annulus
$$
\{z \in \C \mid 1.5 \le \abs{z}\le 2 \}.
$$
\item
Replace the $m$ outgoing edges of $\gamma$ by the annulus
$$
\{z \in \C \mid 1 \le \abs{z}\le 1.5 \}.
$$
\item
If an interior edge $e$ starts at $v_1$ and ends at $v_2$, glue the metrised ribbon graphs $RG_{v_1}$, $RG_{v_2}$ onto $e$, using the parametrisation on the boundaries of  $RG_{v_1}$ and $RG_{v_2}$. Note the outgoing boundary components of a ribbon graph are parametrised in the opposite sense to the incoming boundaries.  Similarly, glue the incoming and outgoing edges to the ribbon graphs at the corresponding vertices.
\end{enumerate}
Then, $\Sigma(\gamma)$ is a topological surface with a piecewise flat (but singular) metric.  The conformal structure attached to this metric extends across the singularities in a natural way.  The boundaries of $\Sigma(\gamma)$ have a natural an analytic parametrisation, and thus $\Sigma(\gamma) \in \mc S(n,m)$.

The map $MG'(n,m) \to \mc S(n,m)$ is easily seen to be compatible with composition.
The fact that it's a weak equivalence comes from the fact that $MG'(n,m) \simeq MG(n,m) \simeq G(n,m)$, and the corresponding map in the homotopy category  $G(n,m) \to \mc S(n,m)$ is homotopic to the usual way of associating a Riemann surface to a ribbon graph.
\end{proof}

\section{Definitions about piecewise linear spaces}
\label{strat}

We would like to construct a chain model of the prop $MG$, by applying a chain functor with certain properties.   The chain functor we need will be defined on a certain category of spaces with PL stratifications.   In this section we will describe this category, and then construct the desired chain complex functor. 

\subsection{Piecewise linear stratifications}

\begin{defn}
Let $X$ be a locally compact Hausdorff topological space.  A \emph{piecewise linear stratification} $S$ of $X$ is a
decomposition of $X$ into a locally finite collection of disjoint
connected pieces $X_\alpha$, called the strata of $X$.  The closure $\br X_\alpha$ is required to be a disjoint union of strata of $X$.

Each $\br X_\alpha$ is a subset of some $\R^{n_\alpha}$ cut out by a finite number of linear inequalities (some of the inequalities might be strict, others may be non-strict).  The region where all of the inequalities are strict is $X_\alpha$.

If $X_\beta$ is a stratum with $X_\beta \subset \br X_\alpha$, then the map $\br X_\beta \to \br X_\alpha$ is required to be affine linear.  That is, it is the restriction of an affine linear map $\R^{n_\beta} \to \R^{n_\alpha}$.

Let $(X,S)$, $(Y,T)$ be two spaces with piecewise linear stratifications.  A continuous map  $f :X \to Y$ is said to be compatible with $S$ and $T$ if  every stratum $X_\alpha$ is mapped into some stratum $Y_\beta$, and the map $f \vert_{X_\alpha} : X _\alpha \to Y_\beta$ is affine linear.

A refinement of a PL stratification $S$ of $X$ is a PL stratification $S'$ such that the identity map $(X,S') \to (X,S)$ is compatible with the stratifications.

A \emph{piecewise linear structure} on $X$ is an equivalence class of piecewise linear stratifications, where two PL stratifications $S,S'$ are equivalent if they admit a common refinement.  A topological space with a PL structure will be called a PL space.
\end{defn}

\begin{lem}
Let $(X,S)$, $(Y,T)$ be spaces with PL stratifications.  Let $T'$ be a refinement of $T$.  Then, for any map $f : (X,S) \to (Y,T)$, there exists some refinement $S'$ of $S$ such that $f$ is compatible with the stratifications $S',T'$, that is, gives a map $f : (X,S' ) \to (Y,T')$.
\end{lem}

\begin{definition}
If $X,Y$ are PL spaces, a PL map $f : X \to Y$ is a continuous map such that there exists stratifications $S$, $T$ of $X$ and $Y$, in the given equivalence classes, such that $f$ is compatible with $S$ and $T$.
\end{definition}
\begin{lemma}
If $f : X \to Y$ is a PL map, and $T$ is a stratification of $Y$, then there exists a stratification $S$ of $X$ such that $f : (X,S) \to (Y,T)$ is compatible with the stratifications.  
\end{lemma}
\begin{proof}
By definition, there exists some stratification $T'$ of $Y$ and a stratification $S$ of $X$ such that $f : (X,S ) \to (Y,T')$ is compatible with the stratifications.  Let $T''$ be a common refinement of $T'$ and $T$.  The lemma asserts the existence of a stratification $S'$ of $X$ such that $f : (X,S') \to (Y,T'')$ is compatible with the stratifications. Since the identity map $(Y,T'') \to (Y,T)$ is compatible with the stratifications, so is the map $(X,S') \to (Y,T)$. 
\end{proof} 
\begin{lemma}
The composition of two PL maps is PL. 
\end{lemma}
\begin{proof}

Let $X,Y,Z$ be PL spaces and let $f : X \to Y$, $g : Y \to Z$ be PL maps.  Then, there must exist stratifications $S$ of $X$, $T,T'$ of $Y$ and $U$ of $Z$, such that $f : (X,S) \to (Y,T)$ is compatible with the stratifications, and $g : (Y,T') \to (Z,U)$ is compatible with the stratifications. The previous lemma implies that there exists $S'$ such that $f : (X,S') \to (Y,T')$ is compatible with the stratifications.  It follows that the composed map $(X,S') \to (Z,U)$ is compatible with the stratifications. 
\end{proof} 
\begin{ex}
Notice that the space $G(n,m)$ has a natural piecewise linear stratification. This is given by declaring two graphs in the same stratum, if they are homeomorphic in a way respecting all markings, and the ribbon structure, but not necessary the metric. Each stratum consists of the quadrant of positive numbers in $\mathbb R^k$, where $k$ is the number of edges.

This can be used to define a stratification on $MG(n,m)$. Two meta-graphs are in the same stratum, if they consists of the same connected, directed graph $\gamma$, with labelings $RG_v\in G(v_{in},v_{out})$ in the same stratum at each vertex $v$, and such that the discrepancies at each edge $e$ have the same sign, i.e. the discrepancies are both either positive, negative or zero. Note that for a fixed graph $\gamma$, this is a refinement of the product stratification of $\prod_v G(v_{in},v_{out})$.   We obtain the boundary of a stratum in the case that the corresponding input and output lengths coincide.
\end{ex}

\subsection{Piecewise smooth differential forms and de Rham currents}
Let $(X,S)$ be a space with a PL stratification.   We will
define a complex  $\Omega^\ast(X,S)$ of piecewise smooth
differential form on $X$ with respect to this stratification $S$.
An element $\omega \in \Omega^\ast(X,S)$ is the data of a
differential form $\omega_\alpha \in \Omega^\ast(\br X_\alpha)$, for
each $\alpha \in S$,  such that if $\br X_\beta \subset \br
X_\alpha$,
$$
\omega_\alpha \mid_{\br X_\beta} = \omega_\beta
$$
In other words, $\Omega^\ast(X,S)$ is defined to be the inverse limit of the $\Omega^\ast(\br X_\alpha)$.

If $S'$ is a refinement of $S$, there is a natural map
$$
\Omega^\ast(X,S) \to \Omega^\ast(X,S').
$$
If $X$ is a space with a PL structure, given by an equivalence class of PL stratifications, let
$$
\Omega^\ast(X)  = \limdir \Omega^\ast(X, S)
$$
be the direct limit over all stratifications .  This is the complex of piecewise smooth differential forms on $X$.

If $f : X \to Y$ is a PL map, then there is a pull back map $f^\ast : \Omega^\ast(Y) \to \Omega^\ast(X)$.

\begin{definition}\label{PL-chain-functor}
A \emph{PL chain complex functor} is a functor $C_\ast$ from the category of PL spaces to that of chain complexes over $\R$ with the following properties.
\begin{enumerate}
\item
The homology of $C_\ast(X)$ coincides with the singular homology of $X$, with coefficients in $\R$.
\item\label{Eilenberg-Zilber}
There is a K\"unneth map
$$
C_\ast(X) \otimes C_\ast(Y) \to C_\ast(X \times Y)
$$
which makes $C_\ast$ into a symmetric monoidal functor.
\item
The complex $C_\ast(X)$ has the structure of differential graded $\Omega^\ast(X)$ module, in a functorial way.
\item\label{respect-refinement}
If $S$ is a PL stratification of $X$, compatible with the PL structure, then
$$
C_\ast(X) = \limdir_{\alpha \in S} C_\ast( \br X_{\alpha} ) .
$$
\end{enumerate}
\end{definition}

\begin{prop}\label{C-PL}
There exists a PL chain complex functor.
\end{prop}
\begin{proof}[Proof (sketch)]
Define $C_n(X)$ to be vector space over $\R$ spanned by triples
\begin{enumerate}
\item
$A \subset \R^{n+k}$, a compact subset cut out by a finite number of linear inequalities, such that the interior of $A$ is dense in $A$,
\item
a PL map $f : A \to X$,
\item
a piecewise smooth form $\omega \in \Omega^k(A)$
\end{enumerate}
modulo the following equivalence relations.
\begin{enumerate}
\item
If $\phi : \R^{n+k} \to \R^{n+k}$ is an invertible affine-linear map, then
$$
(A,f,\omega) \sim \pm (\phi(A), f \circ \phi^{-1}, (\phi^{-1})^\ast \omega )
$$
where the sign is determined by whether or not $\phi$ preserves orientation.
\item
If $h : \R^n \to \R$ is an affine-linear function, let $A_+ = \{x \in A \mid h(x) \ge 0\}$, and let $A_- =\{x \in A \mid h(x) \le 0 \}$.  Then we identify
$$
(A,f,\omega) = ( A_+, f \vert_{A_+}, \omega\vert_{A_+}  ) +  ( A_-, f \vert_{A_-}, \omega\vert_{A_-}  ).
$$
\end{enumerate}
The differential is obvious:
$$
\d ( A, f,\omega)  = ( \partial A, f \vert_{\partial A}, \omega\vert_{\partial A}  ) \pm (A,f, \d \omega).
$$

The axioms are not difficult to check.

\end{proof}

\begin{remk}
Notice that by property \eqref{respect-refinement} of Definition \ref{PL-chain-functor}, to give a  chain on $X$, it suffices to give it on some stratification of $X$ in the given equivalence class.   In particular, for a point $x$ inside a stratum of $X$, we can always find a stratification such that this point is a stratum.  Then we can define a zero chain $[x]\in C_0(X)$ by taking the constant function $1$ on the stratum determined by $x$ and zero on all other strata. We will use this construction in section \ref{string-topology} to identify operation on the loop space of a manifold, determined by a ribbon graph $\gamma$.
\end{remk}

\section{A variant of the Hochschild chain complex}\label{modified-hochschild}

In this section, we recall the definition of the Hochschild chain complex and describe quasi-isomorphic models for it. In particular, we construct the space $L_\ast (\mc A)[n]$, which quasi-isomorphic to the $n^{th}$ fold Hochschild chain complex $CH_\ast(\mc A,\mc A)^{\otimes n}$, for which we will construct the structure of a lax algebra in the next section.

\begin{defn}
Let $(\mc A,d,\cdot)$ be a differential graded associative algebra. The \emph{Hochschild chain complex of $\mc A$} is defined to be $CH_\ast(\mc A,\mc A)=\bigoplus_{k\geq 0} \mc A^{\otimes k+1}$. There is a differential $\delta:CH_\ast(\mc A,\mc A)\to CH_\ast(\mc A,\mc A)$, with $\delta^2=0$, given by
\begin{multline*}\label{Hochschild-differential}
 \delta(a_0\otimes \cdots\otimes a_k) = \sum_{j=0}^k (-1)^{\epsilon_{j}}  a_0\otimes\cdots\otimes d(a_j)\otimes  \cdots\otimes a_k\\
\quad \quad\quad\quad\quad\quad\quad +\sum_{j=0}^{k-1} (-1)^{\epsilon_{j}} a_0\otimes\cdots\otimes (a_j\cdot a_{j+1})\otimes  \cdots\otimes a_k \\
  + (-1)^{(|a_k|+1)\cdot \epsilon_{k}} (a_k\cdot a_0)\otimes a_1\otimes \cdots\otimes a_{k-1},\quad\quad\quad
\end{multline*}
where $\epsilon_i=|a_0|+\ldots+|a_{i-1}|+i-1$.
\end{defn}

We are interested in a variant of $CH_\ast(\mc A,\mc A)$. To this end, give the standard $k$-simplex $\Delta_k=\{(t_1,\dots,t_k)| 0\leq t_1\leq \cdots\leq t_k\leq 1\}$ the structure of PL space in the obvious way.  Thus, there is a cosimplicial chain complex whose $k$-simplices are $C_\ast(\Delta_k)$.

\begin{defn}\label{LA-def}
We define the complex $L_\ast (\A)= \bigoplus_{k \geq 0} \mathcal
\A^{\otimes k+1} \otimes C_\ast(\Delta_k) / \sim $, where the tensor product is the completed tensor product,  and where $\sim$  is
the usual equivalence relation generated by,
\begin{equation*}
 ( a_0 \otimes\cdots \otimes a_{i-1}\cdot a_{i}\otimes
\cdots\otimes a_{k})\otimes c
 \sim (a_0 \otimes\cdots
\otimes a_{i-1}\otimes a_{i} \otimes \cdots\otimes a_k )\otimes
(\delta_i)_*c, 
\end{equation*}
where $\delta_i:\Delta_{k-1}\to \Delta_k$ is the inclusion of the $i^{th}$ face for $1\leq i\leq k$, and a similar relation for $\delta_{k+1}$. The differential on $L_\ast (\A)$ is induced from the differential on the tensor product of chain complexes $\A^{\otimes k+1} \otimes C_\ast(\Delta_k)$.
\end{defn}
The complex $L_\ast(\mc A)$ is quasi-isomorphic to the Hochschild chain complex $CH_\ast(\mc A, \mc A)$ of $\A$.  This is rather easy to see; one can define the Hochschild chain complex $CH_\ast(\mc A, \mc A)$ to be the chain complex associated to the simplicial chain complex whose complex of $n$ simplices is $\A^{\otimes n+1}$, and whose face maps are given by taking products.    The complex $L_\ast(\mc A)$ is defined by starting with the same simplicial chain complex and applying a variant of the realization functor.

The following more general proposition will be used in Section \ref{string-topology}. We use the notation $B\A=\bigoplus_{k\geq 0} (s\A)^{\otimes k}$ for the tensor algebra on the shifted space $s\A$, which is given by $\A$ shifted down by one. For a fixed $r\geq 1$, denote by $\A\otimes B\A\otimes \cdots \otimes \A\otimes B\A$, the $r$-fold tensor product of $\A\otimes B\A$. Homogeneous elements of this space are of the form $a_0\otimes (a_1\otimes\cdots\otimes a_{k_1-1})\otimes \cdots \otimes a_{k_{r-1}}\otimes (a_{k_{r-1}+1}\otimes\cdots\otimes a_{k_{r}-1})$, where the $i^{th}$ factor $\A\otimes B\A$, for $1\leq i \leq r$, consists of $k_i$ tensor products, so that the element of $\A$ in the $i^{th}$ factor $\A\otimes B\A$ is at the $k_{i-1}^{th}$ position, where $k_0=0$ is assumed.
This space has a differential $\delta$, given by,
\begin{multline*}
\delta\big(a_0\otimes (a_1\otimes\cdots\otimes a_{k_1-1})\otimes \cdots \otimes a_{k_{r-1}}\otimes (a_{k_{r-1}+1}\otimes\cdots\otimes a_{k_{r}-1})\big)\\
= \sum_{j=0}^{k_r -1} (-1)^{\epsilon_{j}} a_0\otimes \cdots \otimes d(a_j)\otimes \cdots \otimes a_{k_r-1}
+\sum_{j=0}^{k_r-2} (-1)^{\epsilon_{j}}a_0\otimes \cdots\otimes (a_j a_{j+1})\otimes \cdots \otimes a_{k_r-1} \\
+(-1)^{(|a_{k_r-1}|+1)\epsilon_{k_r-1}} (a_{k_r-1}a_0)\otimes a_1\otimes\cdots \otimes a_{k_r-2}, \quad
\end{multline*}
where $\epsilon_i=|a_0|+\cdots+|a_{i-1}|+i-r_i$, where $r_i$ denotes the number of $k_j$'s which are $\leq i$. Note that in the above, we multiply $\A$ in $\cdots\otimes B\A\otimes  \A\otimes B\A\otimes \cdots $ from the left using the rightmost tensor factor of the left $B\A$, and similar using the leftmost factor of the rigth $B\A$, except in the case of the unit $1_{B\A}$.

For any real number $s \geq 0$ consider the magnified simplex $\ s \Delta_{k-1}=\{(t_0, \ldots, t_{k-1}) ~| ~ 0= t_0\leq t_1\leq\cdots \leq t_{k-1}\leq s \}$. For any sequence of integers $0=k_0< k_1< \cdots< k_r=k$, and real numbers  $0=s_0\leq s_1\leq \cdots \leq s_r$, consider the following subset of $s_r \Delta_{k-1}$.  

$$S_{k_0,\ldots,k_r}^{s_1, \ldots, s_r}=\big\{0=t_0\leq t_1\leq \cdots\leq t_{k-1}\leq  s_r  \,\,\, \big| \,\,\, t_{k_j}=s_j \text{ for all  } j=0,\ldots,r-1 \big\}. $$

This is a subset of $s_r \Delta_{k-1}$ obtained by fixing the values of of the $k_j^{th}$ entry to a prescribed number $s_j$. Let us choose a refinement of $s_k \Delta_{k-1}$, such that $S_{k_0,\ldots,k_r}^{s_1, \ldots, s_r}$ defines a stratum.   Let $c_{k_0,\ldots,k_r}^{s_1, \ldots, s_r}  \in C_{k-r}(s_r\Delta_{k-1})$ be the fundamental chain of $S_{k_0,\ldots,k_r}^{s_1, \ldots, s_r}$. More explicitly, this is, in the notations of proposition \ref{C-PL}, a triple $(S_{k_0, \ldots, k_r}^{s_1, \ldots, s_r}\subset \R^{k-r},f:S_{k_0,\ldots,k_r}^{s_1, \ldots, s_r}\hookrightarrow s_r \Delta_{k-1},1\in \Omega^0(S_{k_0,\ldots,k_r}^{s_1, \ldots, s_r}))$.

\begin{prop}\label{LA-quasi-isos}
Given a sequence of numbers $0=s_0\leq s_1\leq \cdots \leq s_r$, there is a quasi isomorphism  $\A\otimes B\A\otimes \cdots \otimes \A\otimes B\A\to \left(\bigoplus_{i\geq 0}\A^{\otimes i+1}\otimes C_\ast(s_r\Delta_i)\right)/\sim$, given by
\begin{multline*}
a_0\otimes (a_1\otimes\cdots\otimes a_{k_1-1})\otimes \cdots \otimes a_{k_{r-1}}\otimes (a_{k_{r-1}+1}\otimes\cdots\otimes a_{k_{r}-1}) \mapsto \\
[a_0\otimes\cdots\otimes a_{k_r-1}\otimes c_{k_0,\ldots,k_r}^{s_1, \ldots, s_r}].
\end{multline*}
\end{prop}
\begin{proof}
From the equivalence relation in Definition \ref{LA-def}, one can see that the above map is a chain map. Let $S_k^{s_1, \ldots, s_r}=\bigcup_{k_0,\ldots, k_r} S_{k_0,\ldots, k_r}^{s_1, \ldots, s_r} $ and define two new complexes $C_1=\left(\bigoplus_{k\geq 0} \mc A^{\otimes k}\otimes C_\ast (S_k^{s_1, \ldots, s_r})\right) / \sim'$ and $C_2=\left(\bigoplus_{k\geq 0} \mc A^{\otimes k}\otimes C_\ast(s_r\Delta_{k-1})\right) / \sim'$, where the equivalence relation $\sim'$ is generated by the relations from Definition \ref{LA-def}, as well as the additional relation 
\begin{equation*}
 ( a_0 \cdots\otimes a_{i-1}\otimes1\otimes a_{i+1}
\otimes \cdots a_{k})\otimes c
 \sim (a_0 \cdots \otimes
a_{i-1}\otimes a_{i+1}\otimes \cdots  a_k )\otimes
(\sigma_i)_*c, 
\end{equation*}
where $\sigma_i:\Delta_{k}\to \Delta_{k-1}, (t_1,\dots, t_k)\mapsto (t_1,\dots, \hat t_i,\dots, t_k)$.  Since the $S_{k_0,\ldots,k_r}^{s_1, \ldots, s_r}$'s are strata of $s_k\Delta_{k-1}$, the canonical map $C_1\to C_2$ is an chain map. This chain map is clearly injective. It is also surjective, since every element in $C_2$ has a representative with degenerate $1$'s added at the fixed positions $s_1,\dots, s_{r-1}$. In short, this map is an strict isomorphism of complexes.  We thus have a commutative diagram,
\begin{equation*}
\xymatrix{  \bigoplus_{k\geq 0} \mc A^{\otimes k}\otimes C_\ast (S_k^{s_1, \ldots, s_r}) / \sim
\ar[r] \ar[d] &  \bigoplus_{k\geq 0} \mc A^{\otimes k}\otimes C_\ast (s_r\Delta_{k-1}) / \sim  \ar[d] \\
 \bigoplus_{k\geq 0} \mc A^{\otimes k}\otimes C_\ast (S_k^{s_1, \ldots, s_r}) / \sim' \ar[r] &  \bigoplus_{k\geq 0} \mc A^{\otimes k}\otimes C_\ast (s_r\Delta_{k-1}) / \sim' }
\end{equation*}
where the vertical maps are quasi-isomorphisms, and the bottom map is an isomorphism. It follows that the top map is also a quasi-isomorphism. To conclude, note that the above complex $\A\otimes B\A\otimes \cdots \otimes \A\otimes B\A$ is isomorphic to a model of $\bigoplus_{k\geq 0} \mc A^{\otimes k}\otimes C_\ast (S_k^{s_1, \ldots, s_r}) / \sim$, in which the de Rham currents $C_\ast (S_k^{s_1, \ldots, s_r})$ are replaced by simplicial chains.
\end{proof}
Note that in the special case $r=1$ and $s_r=1$, we get the following corollary.
\begin{cor}
$L_\ast(\mc A)$ is quasi-isomorphic to $CH_\ast(\A,\A)$.
\end{cor}

We will also be using a variant of $CH_\ast(\A,\A)^{\otimes m}$.
\begin{definition}\label{LA[m]}
Define $L_\ast(\A)[m]$ to be the quotient of
$$
\bigoplus_{l_1,\ldots,l_m}  \A^{\otimes l_1 + 1 } \otimes \cdots \otimes \A^{\otimes l_m + 1 } \otimes C_\ast( \Delta_{l_1} \times \cdots \times \Delta_{l_m} )
$$
by the equivalence relation generated by
\begin{multline*}
\Big(\cdots\otimes ( a^j_0 \otimes\cdots \otimes a^j_{i-1}\cdot a^j_{i}\otimes
\cdots\otimes a^j_{l_j})\otimes \cdots \Big)\otimes c  \\ \sim \Big(\cdots\otimes (a^j_0 \otimes\cdots
\otimes a^j_{i-1}\otimes a^j_{i} \otimes \cdots\otimes a^j_{l_j} )\otimes \cdots \Big)\otimes
(\delta_i^j)_*c,
\end{multline*}
where $\delta_i^j: \Delta_{l_1} \times \cdots \times \Delta_{l_j-1}\times\cdots \times \Delta_{l_m} \to  \Delta_{l_1} \times \cdots\times \Delta_{l_j}\times\cdots \times \Delta_{l_m} $ is the inclusion of the $i^{th}$ face in the $j^{th}$ simplex.
The differential on $L_\ast (\A)$ is induced from the differential on the tensor product chain complexes $\A^{\otimes k} \otimes C_\ast(\Delta_k)$.

Note that in the case $m=1$, we recover $L_\ast(\mc A)[1]=L_\ast(\mc A)$ from Definition \ref{LA-def}.
\end{definition}
Thus, there is a map $ L_\ast(\A)^{\otimes m} \to L_\ast(\A)[m] $ coming from the map
$$
C_\ast( \Delta_{l_1} ) \otimes \cdots \otimes C_\ast ( \Delta_{l_m} ) \to C_\ast( \Delta_{l_1} \times \cdots \times \Delta_{l_m} )
$$
from Definition \ref{PL-chain-functor}\eqref{Eilenberg-Zilber}. This map is a quasi-isomorphism.


\section{Action of moduli spaces of Riemann surfaces}

We are now ready to state and proof our main theorem, which gives a lax algebra over the prop $C_\ast(MG)$ on for the spaces $L_\ast (\mc A)[n]$.    As always, let $\A$ denote a Calabi-Yau elliptic space, of parity $p(\A)$.

\begin{thm}\label{main-thm}
Fix a Hermitian metric on $\A$.  Then, there is a $\Z/2$ graded lax algebra over the prop $C_\ast(MG)$ sending
$$
n \mapsto L_\ast(\A)[n].
$$
Thus, there are chain maps
$$
C_\ast(MG(n,m) ) \otimes L_\ast(\A)[n] \to L_\ast(\A)[m]
$$
compatible with differentials, composition, and tensor product maps.

Further, this structure of lax $C_\ast(MG)$ algebra structure on $L_\ast(\A)$ independent, up to homotopy, on the choice of Hermitian metric on $\A$.  More precisely, if we have a smooth family of Hermitian metrics on $\A$ parametrized by $\Delta_k$, then there are chain maps
$$
C_\ast(MG(n,m)) \otimes L_\ast(\A)[n] \to L_\ast(\A)[m] \otimes \Omega^\ast(\Delta_k)
$$
making giving $L_\ast(\A)$ a family of lax $C_\ast(MG)$ algebra structures over the differential graded algebra $\Omega^\ast(\Delta_k)$.  
\end{thm}
Although this structure apparently depends on the choice of Hermitian metric on $\A$ (e.g. the metric on the manifold in the string topology example), later we will see that the if we vary the metric then the structure of $C_\ast(MG)$ algebra changes by a homotopy.  Since the space of Hermitian metrics on $\A$ is contractible in all examples, the $C_\ast(MG)$ structure is well-defined up to contractible choice.  
\begin{proof}
To keep the notation simple, we will describe what happens when $p(\A)$ is even.  The proof in the other case is similar.  

We want to define maps
$$
C_\ast(MG(n,m))\otimes L_\ast
(\A)[n]\to L_\ast (\A)[m]
$$
To do this, we will give maps
$$
\A^{\otimes k_1 + 1} \otimes \cdots \otimes \A^{\otimes k_n + 1}  \otimes C_\ast ( \Delta_{k_1} \times \cdots \times \Delta_{k_n} \times MG(n,m)  ) \to  L_\ast (\A)[m]
$$

\subsection*{Step 1}

Let us define a map
$$
\Phi : \Delta_{k_1} \times \cdots \times \Delta_{k_n} \times MG(n,m) \to \Gamma ( n + \sum k_i , m )
$$
as follows.  Let $x_i \in \Delta_{k_i}$, for $1 \le i \le n$, and let  $\gamma \in MG(n,m)$.  Thus, $\gamma$ is a directed graph, with $n$ incoming and $m$ outgoing external edges, each of whose vertices $v$ is labeled by a ribbon graph $RG_v \in G(v_{in}, v_{out})$.

Recall that we can think of $\tr_k$ as the set of  $\{ 0 \le t_1 \le
t_2 \le \ldots \le t_k \le 1\}$.  Alternatively, we can think of
$\tr_k$ as the set of $k$ labeled points on $S^1$, which are  in
cyclic order, positioned in such a way that the base point $1$ on
$S^1$ lies between the $k^{th}$ and the $1$st point.

Let $v$ be a vertex (of any level), and suppose $v$ is attached to the $i^{th}$ incoming external edge of $\gamma$.  The ribbon graph $RG_{v}$ has an incoming boundary component corresponding to this $i^{th}$ incoming external edge of $\gamma$.  This incoming boundary component is parametrised.  Let us replace the ribbon graph $RG_v$ by the ribbon graph with $k_i$ additional incoming external edge of length $0$ attached to this incoming boundary component.  The positions of these external edges are determined by the point $x_i \in \tr_{k_i}$, where we identify $\tr_{k_i}$ with the set of $k$ labeled points on the circle as above.

Repeat this step for all incoming edges attached to $v$, and all vertices $v$.

Recall the \emph{level} of a vertex $v$ of $\gamma$ is the length directed longest path of edges starting at an incoming external edge and ending at $v$.   Every level one vertex is attached to an incoming external edge of $\gamma$, but vertices of levels greater than one can be attached to these incoming external edges also.

The next step is to glue these ribbon graphs $RG_v$ together. We will start by gluing the level one vertices to the level two, and the level two to the level three, and so on.

Let  $v$ be a vertex of level $2$.  Let $e$ be an edge of $\gamma$ which goes from $v'$ to $v$. Thus, $v'$ is a vertex of level $1$.  There is an outgoing boundary component of $RG_{v'}$ and incoming boundary component of $RG_{v}$ attached to $e$. Rescale uniformly this boundary component  of $RG_{v'}$ until it's the same length as the corresponding boundary component of $RG_v$.  Then, glue these two parametrised boundary components together.

Repeat this procedure for edges joining vertices of level two to level three, and so forth.  We end up with a ribbon graph
$$
\Phi (x_1, \ldots , x_k , \gamma  )  \in \Gamma ( n + \sum k_i, m ).
$$

\subsection*{Step 2}
Let
$$
X( k_1,\ldots, k_n ; m ) \subset \Gamma ( n + \sum k_i, m)
$$
be the image of $\Phi$.   $X(k_1,\ldots,k_n;m)$ is the space of all metrised ribbon graphs satisfying certain properties.    Let us stratify $X(k_1,\ldots,k_n;m)$ in the natural way, so that the strata are the intersections of the strata of $\Gamma ( n + \sum k_i, m)$ with $X(k_1,\ldots,k_n;m)$.

Let $\gamma \in X(k_1,\ldots,k_n;m)$ be a ribbon graph.  Then the closure of the stratum containing $\gamma$ is $\Met(\gamma)$, the space of all allowable metrics on $\gamma$. From now on, we will think of $\gamma$ as a ribbon graph \emph{without} a metric.

The ribbon graph $\gamma$ has $m$ outgoing external edges.  Each of these is attached to a boundary component of the ribbon graph which is an embedded circle.   These boundary circles are disjoint, and each has precisely one external edge.   Let us call these boundary components of $\gamma$ outgoing, and the remaining boundary components incoming.

The external edge on these boundary components serves as a base point. We orient these outgoing boundary components in the opposite way to the orientation of the incoming boundary components, i.e.\ in the opposite of the natural orientation that exists on a boundary of any ribbon graph. If we choose a metric on $\gamma$, the choice of orientation and of basepoint gives a way to parametrise these outgoing boundary circles.

Each of the $m$ outgoing boundary circle of $\gamma$ has two sides. One, which we call outgoing, is the one with the single external outgoing edge attached to it. The other has some number $l_i$ ($1 \le i \le m$) of edges attached to it.

The position of these $l_i$ marked points on the $i^{th}$ boundary circle gives a map
$$
\Met(\gamma) \to \Delta_{l_i}.
$$

Let $\gamma'$ be the ribbon graph obtained by removing these $m$ boundary circles.  There is a map
$$
\Met(\gamma) \to \Met(\gamma').
$$
In fact, we can identify
$$
\Met(\gamma) = \Met(\gamma') \times \Delta_{l_1} \times \cdots \times \Delta_{l_m} \times \R_{> 0}^m
$$
where the factor of $\R_{> 0}^m$ corresponds to the lengths of the outgoing boundary circles.

\subsection*{Step 3}
Because of the push-forward map
$$C_\ast( \Delta_{k_1} \times \cdots \times \Delta_{k_n}  \times MG(n,m) ) \to C_\ast(X (k_1,\ldots,k_n;m ) ) $$
it suffices to define maps
$$
\A^{\otimes k_1 +1 } \otimes \cdots \otimes \A^{\otimes k_n + 1} \otimes C_\ast(X (k_1,\ldots,k_n;m ) ) \to L(\A)[m].
$$
Because
$$
 C_\ast(X (k_1,\ldots,k_n;m ) )  = \limdir C_\ast(\Met(\gamma))
$$
it suffices to define maps
$$
\A^{\otimes k_1 +1 } \otimes \cdots \otimes \A^{\otimes k_n +1} \otimes C_\ast(\Met(\gamma) ) \to L(\A)[m]
$$
and check that they are compatible with the direct limit.

The construction of \cite{C3} gives us a map
$$
K : \A^{\otimes k_1 +1 } \otimes \cdots \otimes \A^{\otimes k_n +1}  \to  \Omega^\ast(\Met(\gamma')) \otimes \A^{\otimes l_1} \otimes \cdots \otimes  \A^{\otimes l_m} .
$$
Pulling back via the map $\Met(\gamma) \to \Met(\gamma')$, we get maps
$$
\A^{\otimes k_1 +1 } \otimes \cdots \otimes \A^{\otimes k_n +1} \to  \Omega^\ast(\Met(\gamma)) \otimes \A^{\otimes l_1} \otimes \cdots \otimes  \A^{\otimes l_m}.
$$
Now, the cap product action of $\Omega^\ast(\Met(\gamma))$ on $C_\ast(\Met(\gamma)$ gives maps
\begin{multline*}
\A^{\otimes k_1 +1 } \otimes \cdots \otimes \A^{\otimes k_n +1}  \otimes C_\ast(\Met(\gamma)) \\ \to  \Omega^\ast(\Met(\gamma)) \otimes C_\ast(\Met(\gamma)) \otimes  \A^{\otimes l_1} \otimes \cdots \otimes  \A^{\otimes l_m} \\ \to C_\ast(\Met(\gamma)) \otimes  \A^{\otimes l_1} \otimes \cdots \otimes  \A^{\otimes l_m}.
\end{multline*}
Now, we can push forward along the map $\Met(\gamma) \to \Delta_{l_1} \times \cdots \times \Delta_{l_m} $, to get a map
$$
\A^{\otimes k_1 +1 } \otimes \cdots \otimes \A^{\otimes k_n +1}  \otimes C_\ast(\Met(\gamma)) \to C_\ast (  \Delta_{l_1} \times \cdots \times \Delta_{l_m} ) \otimes  \A^{\otimes l_1} \otimes \cdots \otimes  \A^{\otimes l_m} .
$$
Now, there is a natural map
$$ C_\ast (  \Delta_{l_1} \times \cdots \times \Delta_{l_m} ) \otimes \A^{\otimes l_1 + 1} \otimes \cdots \otimes \A^{\otimes l_k+1} \to L_\ast(\A)[m] .$$
Thus, to finish the construction, we need to give a map
$$
\A^{\otimes l_i } \to \A^{\otimes l_i + 1}.
$$
This is simply the map
$$
a_1 \otimes \cdots \otimes a_{l_1} \to 1 \otimes a_1 \cdots \otimes a_{l_i} .
$$

This construction is automatically compatible with the differentials, because the construction of \cite{C3} is.

\vspace{5pt}

It remains to check two things.
\begin{enumerate}
\item
If $\gamma_1,\gamma_2$ are marked ribbon graphs such that $\Met(\gamma_1) \subset \Met(\gamma_2)$, then the  diagram
$$
\xymatrix{  \A^{\otimes k_1 +1 } \otimes \cdots \otimes \A^{\otimes k_n +1} \otimes C_\ast(\Met(\gamma_1) )   \ar[d] \ar[dr] &   \\
 \A^{\otimes k_1 +1 } \otimes \cdots \otimes \A^{\otimes k_n +1} \otimes C_\ast(\Met(\gamma_2) ) \ar[r] &   L(\A)[n]   }
$$
commutes.
\item
The operations are compatible with the composition maps $C_\ast(MG(n,m))\otimes C_\ast(MG(m,k) ) \to C_\ast(MG(n,k))$.
\end{enumerate}
Both of these points are quite straightforward to check.

\vspace{5pt}
The structure of $C_\ast(MG)$ algebra on $L_\ast(\A)$ we have constructed depends on the choice of Hermitian metric on $\A$.  Next we will see that, up to homotopy, this algebra structure is independent of the metric.  This is rather straightforward, using the results of \cite{C3}.  Suppose one had a smooth family of Hermitian metrics on $\A$, parametrized by a simplex $\Delta_k$.  Then, in \cite{C3}, it is shown that the structure of open TCFT on $\A$ lifts to  a family of open TCFTs, over the differential graded ring $\Omega^\ast(\Delta_k)$.  More precisely, the chain maps 
$$
K : \A^{\otimes n} \to \Omega^\ast(\Gamma(n,m) \otimes \A^{\otimes m}
$$
lift to chain maps 
$$
K : \A^{\otimes n} \to \Omega^\ast(\Gamma(n,m) \otimes \A^{\otimes m} \otimes \Omega^\ast(\Delta_k)
$$
when we have a family of metrics depending on $\Delta_k$.  

The only time the Hermitian metric was used in this paper was when we invoked the construction of \cite{C3}.  It follows that if we have a smooth family of Hermitian metrics on $\A$, parametrized by $\Delta_k$, we get chain maps
$$
C_\ast( MG(n,m) ) \otimes L_\ast(\A) [n] \to L_\ast(\A) [m] \otimes \Omega^\ast(\Delta_k)
$$
giving a family of lax $C_\ast(MG)$ algebra structures on $L_\ast(\A)$, over the differential graded algebra $\Omega^\ast(\Delta_k)$.

\vspace{5pt}

Recall the above construction is for the case when $p(\A)$ is even.
When $p(\A)$ is odd, we need to modify things a little bit, to
include the determinental local system.  We end up with operations
from $C_\ast(G(n,m), \det^{p(\A)})$.  This space also forms a
prop.  However, the local system $\det$ is trivial on $G(n,m)$ in a
way respecting the composition maps. So there is a $\Z/2$ graded
isomorphism of props
$$
C_\ast(G(n,m), \det^{p(\A)}) \iso C_\ast(G(n,m)).
$$
This is not a $\Z$ graded isomorphism, as $\det$ is a graded local system.
(The reason we need to worry about $\det$ at all is that it can not be trivialized  on $\Gamma(n,m)$) in a way compatible with the ``open-string'' gluing).

\end{proof}

\begin{cor}\label{homology-cor}
After passing to homology the above lax algebra structure yields a homological field theory on the Hochschild homology. That is, there is the following map of props,
$$H_\ast(\mc S(n, m)) \to \text{Hom} (HH_\ast (\mc A, \mc A)^{\otimes n}, HH_\ast (\mc A, \mc A)^{\otimes m} ).$$  
This map is independent of the choice of Hermitian metric on $\mc A$, as long as this space of Hermitian metrics is path-connected (which it is in all examples). 
\end{cor}

If we start with a simply connected, compact, oriented, Riemannian manifold $M$, and let $\mc A=\Omega^\ast M$, the de Rham forms on $M$, then it is well known that $HH_\ast(\mc A, \mc A)$ is isomorphic to the cohomology $H^\ast (\mc L M)$ of the free loop space, $\mc L M$, of $M$. We thus get the next corollary.

\begin{cor}\label{string-topology-cor}
For a simply connected, compact, oriented, Riemannian manifold $M$, there is a map of props, $$ H_\ast(\mc S(n, m)) \to \text{Hom} (H^\ast (\mc L M)^{\otimes n}, H^\ast (\mc L M)^{\otimes m} ).$$
\end{cor}

\begin{remk}
We will show in Theorem \ref{string-topology-thm} below, that the action from Corollary \ref{string-topology-cor} for a graph $\gamma\in G(n,m)$ whose chords are trees, coincides with the existing string topology operation on the homology of the free loop space $\mc L M$, induced by $\gamma$. Theorem \ref{main-thm} is thus a chain level version of Corollary \ref{string-topology-cor}, extending as well as lifting the string topology operations to a structure of a lax algebra over the prop of chains on the moduli space of Riemann surfaces. A recent extension of the string topology operations to the homology of the moduli space has been given by V. Godin, \cite{Go}.
\end{remk}

\begin{ex}
We would like to demonstrate the above proof in a concrete example. Consider the meta-graph in $\gamma\in MG(3,2)$ from Figure \ref{meta-graph}.
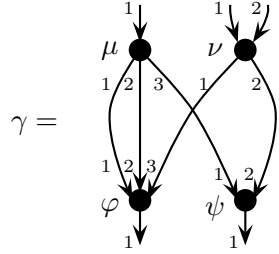
\begin{figure}
\[
\begin{pspicture}(-1,0)(3,3.5)
 \psdot[dotsize=.3](.8,.8)    \psdot[dotsize=.3](2.2,.8)
 \psdot[dotsize=.3](.8,2.8)    \psdot[dotsize=.3](2.2,2.8)
 \rput(.4,2.8){$\mu$}      \rput(1.8,2.8){$\nu$}
 \rput(.4,.7){$\varphi$}     \rput(1.8,.7){$\psi$}
 \rput(-.6,1.8){$\gamma =$}
 \pscurve[arrows=->, arrowsize=.2](.8,2.8)(.8,2)(.8,.9)
 \pscurve[arrows=->, arrowsize=.2](.8,2.8)(.4,2)(.7,.8)
 \pscurve[arrows=->, arrowsize=.2](.8,2.8)(1.5,2)(2.1,.8)
 \pscurve[arrows=->, arrowsize=.2](2.2,2.8)(2.6,2)(2.3,.8)
 \pscurve[arrows=->, arrowsize=.2](2.2,2.8)(1.5,2)(.9,.8)
 \pscurve[arrows=->, arrowsize=.2](.8,.8)(.8,.6)(.8,.2)
 \pscurve[arrows=->, arrowsize=.2](2.2,.8)(2.2,.6)(2.2,.2)
 \pscurve[arrows=->, arrowsize=.2](.8,3.4)(.8,3)(.8,2.9)
 \pscurve[arrows=->, arrowsize=.2](2,3.4)(2,3.3)(2.1,2.9)
 \pscurve[arrows=->, arrowsize=.2](2.5,3.4)(2.5,3.3)(2.3,2.9)
 \rput(.65,3.35){\tiny$1$} \rput(1.85,3.35){\tiny$1$} \rput(2.35,3.35){\tiny$2$}
 \rput(.35,2.35){\tiny$1$} \rput(.65,2.35){\tiny$2$} \rput(1.05,2.35){\tiny$3$}
 \rput(1.65,2.35){\tiny$1$} \rput(2.35,2.35){\tiny$2$}
 \rput(.35,1.25){\tiny$1$} \rput(.65,1.25){\tiny$2$} \rput(.95,1.25){\tiny$3$}
 \rput(1.85,1.15){\tiny$1$} \rput(2.25,1.15){\tiny$2$}
 \rput(.65,.25){\tiny$1$} \rput(2.05,.25){\tiny$1$}
\end{pspicture}
\]
\caption{A meta-graph in $MG(3,2)$ with two levels, to be read from top to bottom. The outputs from $\mu$ and $\nu$ are plugged into the inputs from $\varphi$ and $\psi$.}
\label{meta-graph}
\end{figure}
This graph has two levels; the vertices in level $1$ are labeled by the graphs $\mu\in G(1,3)$ and $\nu\in G(2,2)$ whereas the vertices in level $2$ are labeled by the graphs $\varphi\in G(3,1)$ and $\psi\in G(2,1)$.  As indicated by the meta-graph, the first, second and third outputs of $\mu$ are connected to the first and second inputs of $\varphi$ and the first input of $\psi$, respectively.  Similarly, the first and second output of $\nu$ are connected to the third input of $\varphi$ and the second input of $\psi$, respectively. The graphs for $\mu$, $\nu$, $\varphi$ and $\psi$ are given in Figures \ref{mu}-\ref{psi}.

\begin{figure}
\[
\begin{pspicture}(0,-2)(6,3)
  \pscircle[linestyle=dotted](5,1.5){.9}  \pscircle[linestyle=dotted](2,1.5){.9}
  \pscircle[linestyle=dotted](3,-1){.9}
  \psframe[linecolor=white, fillstyle=solid, fillcolor=white](2.6,1.4)(3,1.6)
  \psframe[linecolor=white, fillstyle=solid, fillcolor=white](4,1.4)(4.4,1.6)
  \psframe[linecolor=white, fillstyle=solid, fillcolor=white](2.9,-.4)(3.1,0)
  \psframe[linecolor=white, fillstyle=solid, fillcolor=white](4.25,.85)(4.4,1)
  \psframe[linecolor=white, fillstyle=solid, fillcolor=white](3.6,-.65)(3.9,-.3)
  \pscircle(5,1.5){.8}  \pscircle(2,1.5){.8} \pscircle(3,-1){.8}
   \rput(.5,.2){$\mu=$}            \rput(2,1.5){\tiny$\mu_1$}
   \rput(5,1.5){\tiny$\mu_2$} \rput(3,-1){\tiny$\mu_3$}
   \rput(4.4,2.4){\tiny$\mu^1$}
  \pscurve(2.8,1.5)(3.3,1.4)(3.5,1.2)(3.7,.73) \pscurve(3.77,.45)(3.8,0)(3.68,-.6)
  \pscurve(4.2,1.5)(3.7,1.4)(3.5,1.2)
  \psline(3.5,1.2)(3,.4) \psline(3,.4)(3,-.2)
  \pscurve(3,.4)(4,.7)(4.4,1)
  \pscurve[linestyle=dotted](3.1,-.1)(3.15,.25)(3.8,.45)(4.4,.8)
  \pscurve[linestyle=dotted](2.9,-.1)(2.9,.5)(3.3,1.15)(2.9,1.4)
  \pscurve[linestyle=dotted](2.9,1.6)(3.2,1.6)(3.5,1.4)(3.8,1.6)(4.1,1.6)
  \pscurve[linestyle=dotted](3.6,-.3)(3.7,0)(3.7,.4)
  \pscurve[linestyle=dotted](3.8,-.6)(3.9,0)(3.9,.4)
  \pscurve[linestyle=dotted](4.1,1.4)(3.65,1.2)(3.8,.8)
  \pscurve[linestyle=dotted](4.2,1)(3.7,.7)(3.24,.65)(3.35,.7)(3.4,.8)(3.5,.9)(3.6,.7)
  \psline[linewidth=.1](5,.7)(5,1)
  \psline[linewidth=.1](3,-.2)(3,-.5)
  \psline[linewidth=.1](2.8,1.5)(2.5,1.5)
  \psline[linewidth=.1](3,.1)(2.7,.1)
\end{pspicture}
\]
\caption{The graph $\mu\in G(1,3)$}
\label{mu}
\end{figure}
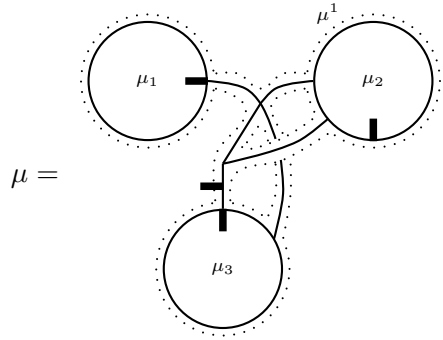
Here we have labeled the $i^{th}$ output components of a graph $\rho=\mu, \nu, \varphi, \psi$ using lower indices $\rho_i$. The $j^{th}$ input component of a graph is labeled by an upper index $\rho^j$, and the input boundary components are traced by a dotted line. For example, $\mu^1$ is the only input component of $\mu$, which is completely traced by the dotted line. Finally, the input and output marked points were denoted by thick external edges. These external edges are of length zero.
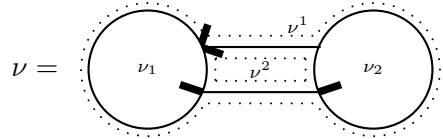
\begin{figure}
\[
\begin{pspicture}(0,0)(6,3)
 \pscircle[linestyle=dotted](5,1.5){.9}  \pscircle[linestyle=dotted](2,1.5){.9}
 \psframe[linecolor=white, fillstyle=solid, fillcolor=white](2.6,1.06)(3,1.94)
 \psframe[linecolor=white, fillstyle=solid, fillcolor=white](4,1.06)(4.4,1.94)
 \psline[linewidth=.1](2.7,1.8)(2.8,2.1)
 \psline[linewidth=.1](2.7,1.8)(3,1.7)
 \psline(2.7,1.8)(4.3,1.8)  \psline(2.7,1.2)(4.3,1.2)
 \pscircle[fillstyle=solid,fillcolor=white](2,1.5){.8}
 \pscircle(5,1.5){.8}
  \rput(5,1.5){\tiny$\nu_2$}  \rput(2,1.5){\tiny$\nu_1$}
 \rput(.5,1.5){$\nu=$}
 \psline[linestyle=dotted](2.9,1.35)(4.1,1.35)  \psline[linestyle=dotted](4.1,1.65)(2.9,1.65)
 \rput(3.5,1.5){\tiny$\nu^2$}  \rput(4,2.1){\tiny$\nu^1$}
 \psline[linestyle=dotted](2.8,1.05)(4.2,1.05) \psline[linestyle=dotted](2.8,1.95)(4.2,1.95)
 \psline[linestyle=dotted](2.9,1.65)(2.9,1.35)  \psline[linestyle=dotted](4.1,1.65)(4.1,1.35)
 \psline[linewidth=.1](4.27,1.2)(4.57,1.3)
 \psline[linewidth=.1](2.73,1.2)(2.43,1.3)
\end{pspicture}
\]
\caption{The graph $\nu\in G(2,2)$}
\label{nu}
\end{figure}
\begin{figure}
\[
\begin{pspicture}(0,0)(6,3)
 \pscircle[linestyle=dotted](4.8,1.5){.7}  \pscircle[linestyle=dotted](2,1.5){.9}
 \psframe[linecolor=white, fillstyle=solid, fillcolor=white](2.6,1.06)(3,1.94)
 \psframe[linecolor=white, fillstyle=solid, fillcolor=white](4,1.06)(4.4,1.94)
 \pscircle(4.8,1.5){.6}  \pscircle(2,1.5){.8}
 \pscircle[linestyle=dotted](4.8,1.5){.5}
 \rput(4.8,1.5){\tiny$\varphi^3$}  \rput(2,1.5){\tiny$\varphi_1$}
 \psline(2.73,1.8)(4.3,1.8)  \psline(2.7,1.2)(4.3,1.2)
 \rput(.5,1.5){$\varphi=$}
 \psline[linestyle=dotted](2.9,1.35)(4.1,1.35)  \psline[linestyle=dotted](4.1,1.65)(2.9,1.65)
 \rput(3.5,1.5){\tiny$\varphi^2$}  \rput(4,2.1){\tiny$\varphi^1$}
 \psline[linestyle=dotted](2.8,1.05)(4.2,1.05) \psline[linestyle=dotted](2.8,1.95)(4.2,1.95)
 \psline[linestyle=dotted](2.9,1.65)(2.9,1.35)  \psline[linestyle=dotted](4.1,1.65)(4.1,1.35)
 \psline[linewidth=.1](4.27,1.2)(4.6,1.3)
 \psline[linewidth=.1](2.73,1.2)(2.43,1.3)
 \psline[linewidth=.1](3.1,1.8)(3.1,1.5)
 \psline[linewidth=.1](2,.7)(2,.4)
\end{pspicture}
\]
\caption{The graph $\varphi\in G(3,1)$}
\label{varphi}
\end{figure}
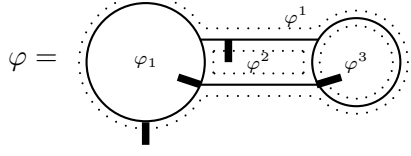

Every edge in any of the graphs $\mu, \nu, \varphi, \psi$ has furthermore a given length. These length were only exhibited in the graph for $\psi$ and suppressed in $\mu$, $\nu$ and $\varphi$ for better readability. We see that the total length for the inputs and output of $\psi$ are as follows:
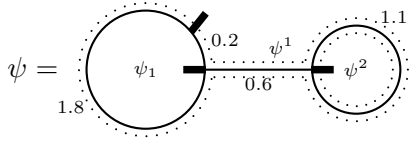
\begin{figure}
\[
\begin{pspicture}(0,0)(6,3)
 \pscircle[linestyle=dotted](4.8,1.5){.7}  \pscircle[linestyle=dotted](2,1.5){.9}
 \psframe[linecolor=white, fillstyle=solid, fillcolor=white](2.6,1.4)(3,1.6)
 \psframe[linecolor=white, fillstyle=solid, fillcolor=white](4,1.4)(4.4,1.6)
 \pscircle(4.8,1.5){.6}  \pscircle(2,1.5){.8}
 \pscircle[linestyle=dotted](4.8,1.5){.5}
 \rput(4.8,1.5){\tiny$\psi^2$}  \rput(2,1.5){\tiny$\psi_1$}
 \rput(.5,1.5){$\psi=$}
  \rput(3.8,1.8){\tiny$\psi^1$}
 \psline(2.8,1.5)(4.2,1.5)
 \psline[linestyle=dotted](2.9,1.4)(4.1,1.4)  \psline[linestyle=dotted](4.1,1.6)(2.9,1.6)
 \psline[linewidth=.1](4.2,1.5)(4.5,1.5)
 \psline[linewidth=.1](2.8,1.5)(2.5,1.5)
 \psline[linewidth=.1](2.6,2)(2.8,2.25)
 \rput(3.05,1.9){\tiny$0.2$}    \rput(3.5,1.3){\tiny$0.6$}
 \rput(5.3,2.2){\tiny$1.1$}    \rput(1,1){\tiny$1.8$}
\end{pspicture}
\]
\caption{The graph $\psi\in G(2,1)$}
\label{psi}
\end{figure}
\begin{eqnarray*}
l(\psi^1)&=& 0.2+ 0.6 + 1.1 + 0.6 + 1.8 = 4.3 \\
l(\psi^2)&=& 1.1\\
l(\psi_1)&=& 1.8 + 0.2 = 2
\end{eqnarray*}
Let's furthermore assume that the length for the boundary components of the other graphs are given by:
\begin{eqnarray*}
& l(\mu^1)= 9.5, \quad l(\mu_1)= 2.1, \quad l(\mu_2)= 1.9, \quad l(\mu_3)= 1.8, \\
& l(\nu^1)= 6.3, \quad l(\nu^2)= 2.5, \quad l(\nu_1)= 1.7, \quad l(\nu_2)= 2.2, \\
& l(\varphi^1)= 4.5, \quad l(\varphi^2)= 1.9, \quad l(\varphi^3)= 1.5, \quad l(\varphi_1)= 2, \\
& l(\psi^1)= 4.3, \quad l(\psi^2)= 1.1, \quad l(\psi_1)= 2.
\end{eqnarray*}
Notice that although $l(\mu_2)=l(\varphi^2)$, we cannot glue $\mu$ to $\varphi$, since $l(\mu_1)\neq l(\varphi^1)$.

Following step $1$ of the proof of Theorem \ref{main-thm}, we will next describe the map $\Phi$ in the current example. In the case that $\Phi:\Delta_4\times \Delta_2\times \Delta_5\times MG(3,2) \to \Gamma (3+4+2+5=14,2)$, we need to find a graph with $14$ incoming external edges and $2$ outgoing external edges, given the data of
\begin{eqnarray*}
 x_1&=&\{0\leq t_{1,1}\leq t_{1,2}\leq t_{1,3}\leq t_{1,4}\leq 1\}\in \Delta_4,\\
 x_2&=&\{0\leq t_{2,1}\leq t_{2,2}\leq 1\}\in \Delta_2, \\
 x_3&=&\{0\leq t_{3,1}\leq t_{3,2}\leq t_{3,3}\leq t_{3,4}\leq t_{3,5}\leq 1\}\in \Delta_5.
\end{eqnarray*}
This is done in the following way. First, attach incoming external edges of length zero to $\mu^1$, $\nu^1$ and $\nu^2$ determined by the rescaled numbers $t_{i,j}$. For example, since $\mu^1$ has a total length of $9.5$, we need to attach incoming edges to the positions $9.5\cdot t_{1,1}, 9.5\cdot t_{1,2}, 9.5\cdot t_{1,3}$ and $9.5\cdot t_{1,4}$. Next, we need to glue the outgoing components of $\mu$ and $\nu$ to the corresponding incoming components of $\varphi$ and $\psi$ determined by the meta-graph $\gamma$ in Figure \ref{meta-graph}. Since the lengths $l(\mu_2)=l(\varphi^2)$, we may immediately identify those boundary components. In the other cases, such as $l(\mu_1)\neq l(\varphi^1)$, we first need to rescale the output $\mu_1$ to the length $l(\varphi^1)=4.5$, and then identify these boundary  components. After repeating this for all attached boundary components, we end up with a graph in $\Gamma(14,2)$ such as in Figure \ref{composed}.
\begin{figure}
\[
\begin{pspicture}(0,-2)(6,3)
 \pscircle(4.8,2){.6}  \pscircle(2,2){.8}
 \psline(2.73,2.3)(4.3,2.3)  \psline(2.73,1.7)(4.3,1.7)
 \pscircle(4.8,-1){.6}  \pscircle(2,-1){.8}
 \psline(2.8,-1)(4.2,-1)
  \psline(2.5,.8)(2,.3) \psline(2,.3)(2,-.2)
  \pscurve(2,.3)(3,.7)(3.4,1.2)(3.6,1.6) \psline(3.66,1.8)(3.8,2.3)
  \pscurve(2.5,.8)(2.9,1.1)(3.2,1.6)       \psline(3.26,1.8)(3.4,2.3)
  \pscurve(2.5,.8)(2.2,1)(2,1.2)
  \pscurve(2.5,.8)(2.7,.74)(2.8,.65) \pscurve(2.9,.55)(3.2,0)(3.4,-1)
  \pscurve(4.2,-1)(4.45,-.8)(4.55,-.55)  \pscurve(4.6,-.35)(4.65,.5)(4.6,1.35)
                        \pscurve(4.58,1.55)(4.55,1.7)(4.3,1.7)
  \pscurve(5.4,-1)(5.15,-.8)(5.05,-.55)  \pscurve(5,-.35)(4.95,.5)(5,1.35)
                        \pscurve(5.0,1.55)(5.1,2.2)(5.2,2.41)
   \psline[linewidth=.1](5.2,2.4)(5.2,2.1)
   \psline[linewidth=.1](5.2,2.4)(4.93,2.35)
   \psline[linewidth=.1](2.8,-1)(2.5,-1)
   \psline[linewidth=.1](2.73,1.7)(2.43,1.8)
   \psline[linewidth=.1](2,0)(1.7,0)
   \psline[linewidth=.02](3.8,-1)(3.8,-.8)
   \psline[linewidth=.02](1.2,2)(1,2)
   \psline[linewidth=.02](3.4,1.2)(3.55,1.05)
   \psline[linewidth=.02](3.4,-1)(3.25,-.85)
   \psline[linewidth=.02](4.8,-1.6)(4.8,-1.4)
   \psline[linewidth=.02](4.95,.5)(4.75,.5)
   \psline[linewidth=.02](4.63,.8)(4.43,.8)
   \psline[linewidth=.02](4.63,.2)(4.43,.2)
   \psline[linewidth=.02](4.8,-.4)(4.8,-.6)
   \psline[linewidth=.02](5,1.2)(5.2,1.2)
   \psline[linewidth=.02](5.35,1.8)(5.17,1.86)
\end{pspicture}
\]
\caption{The composed operation in $X(4,2,5;2)$, also denoted by $\gamma$}
\label{composed}
\end{figure}
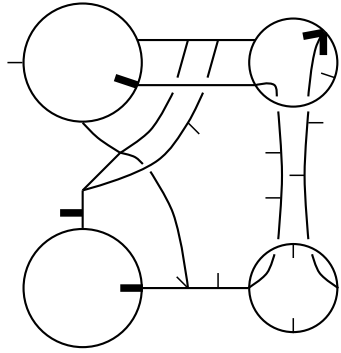
The graph in Figure \ref{composed} depicts a particular stratum of $X(4,2,5;2)$. A different stratum is obtained whenever two vertices in the composed graph come together or cross.

The operation of Theorem \ref{main-thm} consists for the described stratum in Figure \ref{composed} in attaching the Hochschild elements $\mathcal A^{\otimes (4+1)}\otimes \mathcal A^{\otimes (2+1)}\otimes \mathcal A^{\otimes (5+1)}$ at the incoming external edges, and applying the operation from \cite{C3} to the ribbon graph $\gamma'$, which is obtained by removing the two output boundary circles. This is a map $\mathcal A^{\otimes 13}\to \Omega^\ast (\text{Met}(\gamma'))\otimes \mathcal A ^{\otimes 5}$. Pulling back to the stratum in Figure \ref{composed}, we get $\mathcal A^{\otimes 13}\to \Omega^\ast (\text{Met}(\gamma))\otimes \mathcal A ^{\otimes 5}$. The action $\Omega^\ast(\text{Met}(\gamma))\otimes C_\ast(\text{Met}(\gamma))\to C_\ast(\text{Met}(\gamma))$ applied to the initially given chain on $\text{Met}(\gamma)$ induces the map
\begin{eqnarray*}
\mathcal A^{\otimes 14}\otimes C_\ast(\text{Met}(\gamma)) &\to& \Omega^\ast (\text{Met}(\gamma))\otimes C_\ast(\text{Met}(\gamma))\otimes \mathcal A^{\otimes 1} \otimes \mathcal A ^{\otimes 5}\\
&\to & C_\ast(\text{Met}(\gamma))\otimes \mathcal A^{\otimes 6}\\
&\to& C_\ast(\Delta_3\times \Delta_1)\otimes \mathcal A^{\otimes 4}\otimes \mathcal A^{\otimes 2}\\
&\to & L(\mathcal A)[2].
\end{eqnarray*}
The last map comes from pushing $\gamma$ forward to the positions of the vertices attached to the output circles. It is worth mentioning, that since both output marked points are attached to endpoints of $\gamma'$, we don't need to apply the unit $1_{\mathcal A}$ to those output marked points. This concludes the explicit description of Theorem \ref{main-thm} in this example.
\end{ex}

\section{String Topology}\label{string-topology}

We now relate the results of Theorem \ref{main-thm} to the string topology operations on the homology of free loop space of a simply connected, compact, and oriented Riemannian manifold $M$. More precisely, we show that when the elliptic space at hand is the de Rham complex $\mathcal A= \Omega^\ast M$, after passing to the homology, our operations reduce to the string topology operations, originally defined by Chas and Sullivan in \cite{CS1, CS2} and \cite{Su1, Su2}. We use the interpretation and tools used in \cite{CG} for dealing with string topology.

Consider the Calabi-Yau elliptic space $(M,\mathcal A)$, where $\A$ is the space of complex valued de Rham forms of a simply connected, compact, and oriented Riemannian manifold $M$. Thus, the space $\mathcal A= \Omega^\ast (M,\C)$ is the de Rham complex on $M$ with $Q$ the exterior derivation, and trace $\Tr(a)=\int_M a$.  It is well known that the Hochschild homology $HH_\ast(\mathcal A)$ is isomorphic to the cohomology $H^\ast(\mathcal L M)$ of the free loop space, $\mathcal L M$, of $M$. More generally, Definition \ref{LA[m]} implies that $H_\ast(L_\ast(\mathcal A)[n])$ is isomorphic to $H^\ast(\mathcal L M)^{\otimes n}$. It is our aim to identify the operations on the free loop space from Theorem \ref{main-thm} which are given by zero homology classes in $H_0(G(n,m))$. 

The operations we construct depend on the choice of a metric on the Calabi-Yau elliptic space $\Omega^\ast(M)$, but, as we have seen, the homotopy class of the resulting algebraic structure is independent of the metric.

We now examine the operations given by $H_0(G(n,m))$. Let $\gamma\in G(n,m)$ be a ribbon graph such that all its chords are trees. By a chord we mean the complement of the disjoint output circle in the ribbon graph. In order to obtain an operation $[\gamma ]_\ast:H_\ast(L_\ast(\mathcal A)[n])\to H_\ast(L_\ast(\mathcal A)[m])$ from $\gamma$, we start by considering a subdivision of the stratification of $G(n,m)$ such that the point $\gamma$ constitutes a zero dimensional  stratum. This substratification allows us to define a zero chain $[\gamma ]\in C_0 (G(n,m))$ by taking the triple $(\{\gamma\},\{\gamma\}\hookrightarrow G(n,m),1\in\Omega^0\{\gamma\})$, cf. proposition \ref{C-PL}. By Theorem \ref{main-thm}, we get an induced homology operation $L_\ast(\mathcal A)[n]\to L_\ast(\mathcal A)[m]$, and thus an operation on homology, $[\gamma ]_\ast:H_\ast(L_\ast(\mathcal A)[n])\to H_\ast(L_\ast(\mathcal A)[m])$.

For the ribbon graph $\gamma\in G(n,m)$, let $\tilde{\gamma}$ denote the graph obtained by collapsing all the chord pieces to a point $\tilde{\gamma}=\gamma/\sim$.  It was shown in \cite{CG}, that $\tilde{\gamma}$ labels a string topology operation on the homology of the free loop space, $H^\ast(\mathcal L M)^{\otimes n}\to H^\ast(\mathcal L M)^{\otimes m}$. The following theorem states that the operation on $H_\ast(L_\ast(\mathcal A)[n])$ given by a ribbon graph $\gamma\in G(n,m)$ with contractible chords coincides, as a zero dimensional homology class in the space of graphs, with that given by $\tilde{\gamma}=\gamma/ \sim$ in \cite{CG} from string topology.

\begin{thm}\label{string-topology-thm}
In the above notation, the following diagram commutes up to sign, where the vertical maps are isomorphisms:
\begin{equation*}\label{commut-diag-for-string-topology}
\xymatrix{  H_\ast(L_\ast(\mc A)[n])
\ar[rr]^{[\gamma ]_\ast} \ar[d] && H_\ast(L_\ast(\mc A)[m])  \ar[d] \\
 H^\ast(\mc L M)^{\otimes n}  \ar[rr]^{\text{operation from } \tilde{\gamma}} && H^\ast(
\mc L M)^{\otimes m} }
\end{equation*}
\end{thm}
\begin{proof}
Without loss of generality, we will assume that each chord is not only a tree, but furthermore of star shape, as in Figure \ref{star}. This may be assumed, since as a zero chain every tree is homologous to one which is star shaped and the homology operations from Theorem \ref{main-thm} and from \cite[section 2]{CG} only depend on the zero homology class the graphs define.
\begin{figure}
\[
\begin{pspicture}(0,0)(2,2)
 \psdot(1,1)
 \psline(1,0)(1,2)          \psdot(1,2)    \psdot(1,0)
 \psline(.6,.6)(1.6,1.6)  \psdot(.6,.6)  \psdot(1.6,1.6)
 \psline(.2,1)(1.5,1)      \psdot(.2,1)   \psdot(1.5,1)
 \psline(.4,1.6)(1.8,.2)  \psdot(.4,1.6) \psdot(1.8,.2)
 \end{pspicture}
\]
\caption{A star shaped chord}
\label{star}
\end{figure}
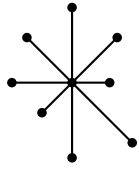

Let us recall again the notation $\A=\Omega^\ast (M,\C)$, which is used throughout this section.
The proof consists in checking the commutativity of the following diagram:
\begin{equation}\label{expanded-commut-diag}
\xymatrix{
H_\ast(L_\ast(\A )[n]) \ar[r]^{[\gamma ]_\ast} &  H_\ast( L_\ast(\A )[m])  \\
H_\ast(B\A  ^{\otimes p} \otimes \A ^{\otimes p}) \ar[r]^{f_\ast} \ar[d]_{\cong} \ar[u]^{\cong} &  H_\ast(B\A  ^{\otimes q} \otimes \A ^{\otimes q})  \ar[d]_{\cong} \ar[u]_{\cong} \\
H_\ast(B\A  ^{\otimes q} \otimes \A ^{\otimes q}) \ar[r]^{g_\ast} \ar[d]_{\cong}  & H_\ast(B\A  ^{\otimes q} \otimes \mathcal C(M^{\times q}))  \ar[d]^{\cong} \\
H^\ast(\mc L M)^{\otimes n}  \ar[r]^{\tilde{\gamma}} &  H^\ast(\mc L M)^{\otimes m} }
\end{equation}
Here, $\mathcal C(M^{\times q})$ denotes the space of currents on $M^{\times q}$, and $p$ and $q$ are determined by the number of special points of the corresponding models for the free loop spaces from Proposition \ref{LA-quasi-isos} with respect to the ribbon graph $\gamma$. Note that all vertical arrows in diagram \eqref{expanded-commut-diag} are isomorphisms.

The proof will proceed in three steps, each corresponding to one of the three squares in diagram \eqref{expanded-commut-diag}. In the first step, we review a convenient quasi-isomorphic model for $L_\ast(\A )[n]$ and $L_\ast(\A)[m]$ from Proposition \ref{LA-quasi-isos}. This model is given by cutting each input and each output of the ribbon graph $\gamma$ into pieces of edges determined by $\gamma$. With this, we attach a copy of the bar construction on $\A$ to intervals and glue them together using one copy of $\A$ at each vertex. Thus, we obtain the space $(B\A ) ^{\otimes p} \otimes \A ^{\otimes p}$ for the input and  $(B\A ) ^{\otimes q} \otimes \A ^{\otimes q}$ for the output of $\gamma$. We then describe the map $f_\ast$, which makes the top diagram in \eqref{expanded-commut-diag} commute.

In the second step, we use a homotopy, which contracts the chord edges in the graph $\gamma$ to a point. Since the operation associated to a chord edge comes from the application of the heat kernel $K(x,y,l)$ at the length $l$ of the edge, and $K(x,y,0)\in \mathcal C(M^{\times 2})$ becomes a current, we see that the contracted operation at $l=0$ has an image involving currents on products of $M$. We obtain the map $g_\ast$, making the middle diagram of \eqref{expanded-commut-diag} commute, by multiplying the forms on the contracted chord edges, and then applying a chain level version of Poincar\'e duality induced by the heat kernel at length zero.

Finally, in the third step, we briefly recall the string topology operation induced by $\gamma$ from \cite{CG}, given by maps as in the bottom row of diagram \eqref{expanded-commut-diag}. We show the commutativity of the bottom diagram of \eqref{expanded-commut-diag} by adapting an argument of Felix and Thomas \cite{FT} on the identification of the string topology product.  

{\bf Step 1:} %
We start by using the model for $L_\ast(\A )$ discussed in Proposition \ref{LA-quasi-isos}. Note that, starting at the initial input marked point, each input component of the graph $\gamma$ can be traced along a sequence of intervals meeting in a linear order at connecting vertices which can be input or output marked points or vertices of the graph. Each such interval is either a piece of one of the output circles or a segment within a chord's edge. Let us refer to the intervals which lie on a chord as chord intervals. Proposition \ref{LA-quasi-isos} applied to this subdivision of the input circle may be viewed as attaching a copy of the tensor algebra $B\A=T(s\A)$ to every interval and a copy of $\A $ to each connecting vertex. Since these intervals have given fixed lengths, we may look at the total length $s^i_1\leq \cdots\leq s^i_{r_i}$, where $s^i_j$ is the sum of the length of the first $j$ intervals on the $i^{th}$ input boundary component starting at the input marked point. There are exactly $r_i$ intervals on the $i^{th}$ input boundary. Note that when both sides of a chord interval appear in the input boundary component, then two different copies of $B\A$ appear, one for each side of the chord interval.

Proposition \ref{LA-quasi-isos} models the $i^{th}$ input of $\gamma$ as $(B\A ) ^{\otimes r_i} \otimes \A ^{\otimes r_i}$. Using Definition \ref{LA[m]}, we obtain $(B\A ) ^{\otimes p} \otimes \A ^{\otimes p}$, where $p=\sum^n_{i=1} r_i$, for the total input $L_\ast(\A )[n]$. Similarly, we may apply Proposition \ref{LA-quasi-isos} to each of the $m$ output boundary components of $\gamma$ to obtain a model $(B\A ) ^{\otimes q} \otimes \A ^{\otimes q}$ for $L_\ast(\A )[m]$. Note, that the difference $p-q$ is exactly the number of chord intervals in $\gamma$. Thus, we obtain quasi-isomorphic models for the domain and range of the top row map in diagram \eqref{expanded-commut-diag}. We also identify the map $f$, which makes the following diagram commute.
\begin{equation*}
\xymatrix{
L_\ast(\A )[n] \ar[r]^{[\gamma ]} &   L_\ast(\A )[m]  \\
B\A  ^{\otimes p} \otimes \A ^{\otimes p} \ar[r]^{f}  \ar[u]^{\cong} &  B\A  ^{\otimes q} \otimes \A ^{\otimes q}   \ar[u]_{\cong} }
\end{equation*}

For an element $\alpha\in B\A  ^{\otimes p} \otimes \A ^{\otimes p}$, we obtain $\alpha\otimes c\in L_\ast(\A )[n] $, where $c$ from Proposition \ref{LA-quasi-isos} is a chain that is given by the constant $0$-form $1$ on a particular substratum  of $\Delta_k$ and zero elsewhere. By construction, this substratification is compatible with the stratification chosen for $[\gamma ]$. The operation defined by $[\gamma ]$ in Theorem \ref{main-thm} applies the map from Theorem \ref{Kevins-map} to the chord pieces and projects the result to the simplicies associated to the outputs. We see that the result of applying the operation from $[\gamma]$ in fact lies in the image of $B\A  ^{\otimes q} \otimes \A ^{\otimes q}$. The map $f$ is obtained by applying Theorem \ref{Kevins-map} to those factors $B\A $ that lie on the chord intervals, and applying the identity on all other tensor factors $B\A$ or $\A$.

{\bf Step 2:} %
In this step, we contract the chord edges to a point. The associated factors of $B\A$ may then be replaced by single factors of $\A$. We then describe the map $g$ making the middle diagram in \eqref{expanded-commut-diag} commute.

For any two choices of positive lengths $l_1$ and $l_2$ of a chord edge, we have a $1$-chain of ribbon graphs whose length on the chord edge varies between $l_1$ and $l_2$, and remains fixed otherwise. Interpreting this $1$-chain as a homotopy between its endpoints shows, that the actions of the chord diagram from Theorem \ref{main-thm}, interpreted via step $1$ as maps $f_{l_1}, f_{l_2}:B\A^{\otimes p}\otimes \A^{\otimes p}\to B\A^{\otimes q}\otimes \A^{\otimes q}$, become equal on homology. We now consider the action induced by varying the length from $0$ to $l_2$. Since the heat kernel becomes singular in this case, we can see that the map $K$ from Theorem \ref{Kevins-map} now takes values in currents $\mathcal C(M^{\times m})$ instead of De Rham forms $\A^{\otimes m}$. If we contract all chord edges for the map $f$ in step $1$, we obtain a homotopy between the action of $f$, and a contracted map $f_0:B\A^{\otimes p}\otimes \A^{\otimes p}\to B\A^{\otimes q}\otimes \mathcal C(M^{\times q})$. Explicitly, the map $f_0$ places the heat kernel $K(x,y,0)$ to the chord edge of length $0$ when the input $1_{B\A}$ is applied on the chord, and vanishes in all other cases.

To obtain the wanted map $g$, we have to complete the following diagram,
\begin{equation}\label{length-to-zero}
\xymatrix{
B\A ^{\otimes p} \otimes \A^{\otimes p} \ar[r]^{f_0\quad} \ar[d]_{\rho} &  B\A  ^{\otimes q} \otimes \mathcal C(M^{\times q}) \ar[d]^{=} \\
B\A ^{\otimes q} \otimes \A ^{\otimes q} \ar[r]^{g\quad}  &  B\A  ^{\otimes q} \otimes \mathcal C(M^{\times q}) }
\end{equation}
First, we define the quasi-isomorphism $\rho:B\A  ^{\otimes p} \otimes \A ^{\otimes p} \to B\A  ^{\otimes q} \otimes \A ^{\otimes q}$, where $q= (p-$number of intervals on the chord edges$)$, which is equal to the total number of intervals for all output components. $\rho$ is defined by
\begin{eqnarray*}
\rho:\cdots \otimes \A\otimes B\A\otimes \A\otimes \cdots &\to & \quad\quad\cdots\otimes \A\otimes \cdots \\
\rho (\cdots\otimes a \otimes b \otimes a' \otimes \cdots)&:= &
\bigg\{ \begin{matrix}
\cdots\otimes (a\cdot a')\otimes \cdots& \text{, if } b=1_{B\A}\\
0 & \text{, otherwise,} 
\end{matrix}
\end{eqnarray*}
where we apply the above to any factor of $B\A$ corresponding to a chord interval of $\gamma$. Thus, in the image of $\rho$, we only get a single copy of $\A$  attached to each chord interval. It may readily be seen that $\rho$ is a chain map, and furthermore a quasi-isomorphism.

Since both $\rho$ and $f_0$ are non-vanishing only on terms $1_{B\A}$ on the chord edges, we may see that it $f_0$ can be factored as $f_0=g\circ\rho$. More explicitly, we define the map $g$ to first multiply all De Rham forms which are attached to the same chord, and apply to this the coproduct given by the heat kernel at length zero the same amount of times. Thus, for a contracted chord with $u$ copies of $\A$ attached, we obtain the operation $\A^{\otimes u}\to \A\to\A^{\otimes u}$ by first multiplying $a_1\otimes \cdots \otimes a_u\mapsto a:= \pm a_1\wedge \cdots \wedge a_u$, and then comultipliying using the heat kernel $a(x)\mapsto \pm \int_M a(x)\wedge K(x,y_1,0)\wedge\cdots\wedge K(x,y_u,0)\in \mathcal C(M^{\times u})$. This is just an explicit description of the action of the map $K$ from Theorem \ref{Kevins-map} at length zero, and thus makes diagram \eqref{length-to-zero} commute. 

Since for each $t\geq 0$, the map $a(x)\mapsto \int_M a(x)K(x,y,t)$ is a chain level representation of the coproduct on cohomology obtained by transporting the coproduct on homology via Poincar\'e duality, we see that $g$ applies product and then coproduct to the De Rham forms which are given at the contacted chords. Furthermore, all the copies of $B\A$ associated to the circle edges remain unchanged, but are only rearranged according to the combinatorics of inputs and outputs of the graph $\gamma$. Using this description, $g$ may now be identified with the string topology operations in the next step.

{\bf Step 3:} %
Let us briefly recall the operations associated to a ribbon graph
with treelike chords from \cite[section 2]{CG}. Let $Maps(\tilde{\gamma},M)$
denote the space of maps from $\tilde{\gamma}$ to $M$, for which each chord
component maps to the same point in $M$. Then, there are two maps
that restrict to the incoming and outgoing components,
respectively, $\rho_{in}:Maps(\tilde{\gamma},M)\to \mc LM^{\times n}$, $\rho_{out}
:Maps(\tilde{\gamma},M)\to \mc LM^{\times m}$. Note that since in this paper we work in a
cohomological rather than a homological setting, our choices of incoming and outgoing
loops are opposite to the ones in \cite{CG}, i.e. the outgoing
components of ${\gamma}$ are disjoint circles. In this notation, we get the first half of the desired map as the induced map on cohomology of $\rho_{in}$, i.e.
$\rho_{in}^\ast:H^\ast(\mc LM)^{\otimes n}\to H^\ast(Maps(\tilde{\gamma},M))$.

Now, note that $\rho_{out}$ is an embedding of infinite dimensional manifolds with finite codimension $c\cdot dim(M)$, which map the chords to a diagonal of marked points in $\mc LM^{\times m}$. We need to use the Umkehr map using the Thom space $Thom(Maps(\tilde{\gamma},M))$ of the normal bundle of $Maps(\tilde{\gamma},M)\hookrightarrow \mc LM^{\times m}$. More precisely,
$Thom(Maps(\tilde{\gamma},M))$ is the one-point compactification of the unit disk
bundle of $Map(\tilde{\gamma},M)$ in $\mc LM^{\times m}$. Then, we obtain the desired
Umkehr map, $\rho^!_{out}$, as a composition of the Thom collapse map $\tau:\mc
LM^{\times m} \to Thom(Maps(\tilde{\gamma},M))$ and the Thom isomorphism $t^\ast$,
$$ \rho_{out}^!: H^\ast(Maps(\tilde{\gamma},M)) \overset{t^\ast}\to H^\ast(Thom(Maps(\tilde{\gamma},M))) \overset{\tau^\ast} \to H^\ast(\mc LM)^{\otimes m}, $$ where $t^*$ is of degree $-c \cdot dim(M)$. With this the bottom row of diagram \eqref{expanded-commut-diag} is defined as the composition of $\rho^\ast_{in}$ with $\rho^!_{out}$,
\begin{equation}\label{CG-maps}
H^\ast(\mc LM)^{\otimes n} \overset{\rho^*_{in}} \longrightarrow H^\ast(Maps(\tilde{\gamma},M)) \overset{\rho^!_{out}} \longrightarrow H^\ast(\mc LM)^{\otimes m}.
\end{equation}
The map \eqref{CG-maps} describes the string topology operation from \cite{CG}. All three spaces used in its definition are fibrations which assign to a map $Y\to M$ its value on some fixed points, determined by the graph $\gamma$,
\begin{eqnarray*}
ev:\mc LM^{\times n} &\to & M^{\times q}, \\
ev:Maps(\tilde \gamma,M)&\to & M^{\times k}, \\
ev:\mc LM^{\times m} &\to & M^{\times q}.
\end{eqnarray*}
Here, we let the fixed points consist of all points, which are either marked points of the inputs or outputs of the ribbon graph $\gamma$, or points, where an intersection occurs, i.e. where a chord is attached. Thus, the cohomological string topology operation given by \eqref{CG-maps} fits into the following commutative diagram,
\begin{equation}\label{sq-in-out}
\xymatrix{
H^*(\mc LM^{\times n}) \ar[r]^{\rho^*_{in}\quad} &  H^*(Maps(\tilde\gamma,M)) \ar[r]^{\quad\rho^!_{out}} & H^*(\mc LM^{\times m}) \\ 
H^*(M^{\times q})\ar[r]^{\Delta_{in}^*} \ar[u]^{ev^*} & H^*(M^{\times k})\ar[r]^{\Delta_{out}^!}\ar[u]^{ev^*} & H^*(M^{\times q})\ar[u]^{ev^*} }
\end{equation}
where the bottom maps $\Delta_{in}:M^{\times k}\to M^{\times q}$ and $\Delta_{out}:M^{\times k}\to M^{\times q}$ represent the diagonal $M\to M\times M$, applied $(k-q)$ times according to the combinatorics of the points identified in the inputs and the outputs of the ribbon graph $\gamma$, and $\Delta_{out}^!$ is the Umkehr map associated to $\Delta_{out}$. Notice, that both $\rho^!_{out}$ and $\Delta^!_{out}$ are maps of degree $-(q-k)dim(M)$.

We model \eqref{sq-in-out} via a method of Felix and Thomas \cite[section 4]{FT} using a Poincar\'e duality model from Lambrechts and Stanley \cite{LS}. This will then be identified with the description from step 2 using the functoriality of the construction. More precisely, take a Poincar\'e duality model $(A=\bigoplus_{i=0}^{dim(M)} A^i,d, \mu,\phi)$ of the manifold $M$, i.e. $(A,d,\mu)$ is a finite dimensional commutative associative, differential graded algebra, quasi-isomorphic to the de Rham forms $\A$, and $\phi:A\to A^*$ is an $A$-bimodule isomorphism of degree $-dim(M)$ representing Poincar\'e duality on homology. Dualizing the coproduct $\mu^*$ on $A^*$ via $\phi$, we obtain the comultiplication $\Delta:A\to A\otimes A$, which is well known to represent a model of the Umkehr map of the diagonal embedding $M\to M\times M$, cf. \cite[section 4]{FT}. More generally, the Umkehr map $\Delta^!_{out}$ is given by a multiple application this comultiplication $\Delta^{\circ (q-k)}:A^{\otimes k}\to A^{\otimes q}$ according to the combinatorics of the outputs of the graph $\gamma$. Clearly, the map $\Delta^*_{in}$ can also  be modeled on $A$ via multiple application of the mutliplication $\mu^{\circ (q-k)}:A^{\otimes q}\to A^{\otimes k}$ according to the combinatorics of the inputs of $\gamma$. 

In order to obtain the top row of diagram \eqref{sq-in-out}, we recall that the two sided (normalized) bar construction $B(A,A,A):=\bigoplus_{n\geq 0} A\otimes (sA/\C)^{\otimes n}\otimes A$ on $A$ is semifree resolution of $A$ as $A$-bimodules, see \cite[Section 2]{FT} for more details. By the property of semifree models, cf. \cite[Chapter 7]{FHT}, it follows, that the input loops (respectively the output loops) are modeled, via the above fibrations, by an identification of the endpoints of $q$ intervals represented by $B(A,A,A)^{\otimes q}$ via multiple tensoring by $A$ over $A\otimes A$. We obtain $B(A,A,A)^{\otimes q}\stackrel {in} \otimes_{A^{\otimes 2q}}A^{\otimes q}$ (respectively $B(A,A,A)^{\otimes q}\stackrel {out} \otimes_{A^{\otimes 2q}}A^{\otimes q}$), where $\stackrel {in} \otimes$ (respectively $\stackrel {out}\otimes$) indicates that the combinatorics of the tensor products are given by the combinatorics of the inputs (respectively outputs) of $\gamma$. Similarly, $Maps(\tilde\gamma,M)$ has a model given by identification of additional $(q-k)$ points, so that we obtain $B(A,A,A)^{\otimes q}\stackrel {\tilde \gamma} \otimes_{A^{\otimes 4q-2k}}A^{\otimes 2q-k}\cong B(A,A,A)^{\otimes q}\stackrel {\tilde \gamma} \otimes_{A^{\otimes 2q}}A^{\otimes k}$. Using again the semifree property of $B(A,A,A)$ applied to the above fibrations, we determine the left square in diagram \eqref{sq-in-out} by applying $(q-k)$ diagonals represented by products $\mu^{\circ(q-k)}$ via $$ A^{\otimes q}\cong A^{\otimes q}\otimes _{A^{\otimes 2(q-k)}}A^{\otimes 2(q-k)}\quad\stackrel {id\otimes \mu^{\circ(q-k)}} \longrightarrow \quad A^{\otimes q}\otimes _{A^{\otimes 2(q-k)}}A^{\otimes (q-k)}\cong A^{\otimes k}, $$ to obtain,
\begin{equation}\label{sq-left}
\xymatrix{
B(A,A,A)^{\otimes q}\stackrel {in} \otimes_{A^{\otimes 2q}}A^{\otimes q}  \ar[rrr]^{id_{B(A,A,A)^{\otimes q}}\otimes \mu^{\circ(q-k)}} &&&  B(A,A,A)^{\otimes q}\stackrel {\tilde\gamma} \otimes_{A^{\otimes 2q}}A^{\otimes k} \\ 
A^{\otimes q}\ar[rrr]^{\mu^{\circ (q-k)}}\ar[u] &&& A^{\otimes k}\ar[u]  }
\end{equation}

Similarly, we obtain the model for the right square in diagram \eqref{sq-in-out}. Note that the methods of \cite[Proposition 5.4]{St} may be used to show that the pullback of a tubular neighborhood of $M^{\times q}\to M^{\times k}$ gives a tubular neighborhood of $\mc LM^{\times m}\to Maps(\tilde\gamma,M)$. This result together with the naturality of the Thom isomorphism $t^*$ shows that the right square in \eqref{sq-in-out} is given by $(q-k)$ comultiplications via
$$ A^{\otimes k}\cong A^{\otimes q}\otimes _{A^{\otimes 2(q-k)}}A^{\otimes (q-k)}\quad\stackrel {id\otimes \Delta^{\circ(q-k)}} \longrightarrow \quad A^{\otimes q}\otimes _{A^{\otimes 2(q-k)}}A^{\otimes 2(q-k)}\cong A^{\otimes q}, $$ to induce the model,
\begin{equation}\label{sq-right}
\xymatrix{
B(A,A,A)^{\otimes q}\stackrel {\tilde\gamma} \otimes_{A^{\otimes 2q}}A^{\otimes k}  \ar[rrr]^{id_{B(A,A,A)^{\otimes q}}\otimes \Delta^{\circ(q-k)}} &&&  B(A,A,A)^{\otimes q}\stackrel {out} \otimes_{A^{\otimes 2q}}A^{\otimes q} \\ 
A^{\otimes k}\ar[rrr]^{\Delta^{\circ (q-k)}}\ar[u] &&& A^{\otimes q}\ar[u]  }
\end{equation}

Thus, we see that the description of the string topology operation from \cite{CG} via the bar construction applies exactly the multiplications $\mu$ and comultiplications $\Delta$ as described in step 2. The only difference to the description in step 2 is, that instead of using the Poincar\'e model $A$, we applied the bar construction to the De Rham forms $\A$, and for the outputs, we used the module given by the currents $\mc C(M^{\times q})$. Now notice, that diagram  \eqref{sq-left} is functorial with respect to change of algebra $(A,d,\mu)\to (A',d',\mu')$. Furthermore, if we treat $A^{\otimes k}$ and $A^{\otimes q}$ as modules over $A$, then diagram \eqref{sq-right} is also functorial with respect to simultaneous change of the algebra, modules and module maps $(A,d,\mu,N_k\stackrel {\Delta_{out}} \longrightarrow N_q)\to (A',d',\mu',N'_k\stackrel {\Delta'_{out}} \longrightarrow N'_q)$. After connecting the Poincar\'e model $A$ with the de Rham forms $\A$, and the comultiplication $A^{\otimes k}\stackrel {\Delta^{\circ (q-k)}}\longrightarrow A^{\otimes q}$ with the application of the heat kernel at time zero $\A^{\otimes k}\stackrel {K(x,y,0)^{(q-k)}}\longrightarrow \mc C(M^\times k)$ as in step 2, we obtain a zig-zag of quasi-isomorphisms, relating both constructions. A spectral sequence argument shows, that these models are in fact quasi-isomormphic. This identifies the string topology operation with the operation described in step 2, and thus completes the proof of the Theorem.
\end{proof}

\section{Higher genus B-model}\label{B-model}

The higher genus B-model was rigorously constructed in \cite{C1,C2} and \cite{K, KS2, KKP}  for any
Calabi-Yau A$_\infty$ category. In this section we will discuss the way in which the explicit construction of this paper, applied to the Calabi-Yau elliptic space constructed from a Calabi-Yau manifold, gives a different, more geometric, construction of the $B$ model at all genera.

\subsection{}

Let $X$ be a smooth projective Calabi-Yau manifold.  Define a dg category $\op{Com}(X)$  whose objects are bounded complexes of (algebraic) vector bundles on $X$, and whose complex of
morphisms $E \to F$ is the Dolbeaut complex 
 $$\Hom_{\op{Com}(X)}^l (E,F) = \oplus_{i-j+k = l} \Omega^{0,i} (X, \Hom (E^i ,F^j ) ).$$   
\begin{theorem}[Keller, \cite{Ke}]
The Hochschild homology of the dg category $\op{Com}(X)$ is $H^{-\ast}(X, \Omega^{\ast}(X) )$.
\end{theorem}
If we fix any object $E \in \op{Com}(X)$, then $\Hom_{\op{Com}(X)}(E,E)$ is a Calabi-Yau elliptic algebra; and the choice of Hermitian metric on $X$ together with a Hermitian metric on the vector bundles making up the complex $E$ leads to a Hermitian metric on this Calabi-Yau elliptic algebra.  

It has been shown in \cite{BV} that the derived category of sheaves on $X$ has a strong generator $E$.  Since $X$ is projective, $E$ can be taken to be a finite complex of vector bundles on $X$, and so an object of $\op{Com}(X)$.   The fact that $E$ is a strong generator implies  that  that the dg category of modules over $\Hom_{\op{Com}(X)}(E,E)$ is equivalent to the category $\op{Com}(X)$.  Morita invariance of Hochschild homology then implies that 
$$
HH_\ast (\Hom_{\op{Com}(X) }(E,E)) = H^{-\ast}(X, \Omega^{\ast}(X) ).
$$

The dg algebra $\Hom_{\op{Com}(X) }(E,E)$ is a Calabi-Yau elliptic algebra. Thus, the construction of this paper gives an explicit action of the chains of the moduli space of Riemann surfaces on the Hodge cohomology of $X$, which is dual to the space
$$
\oplus H^i (X, \wedge^j TX ) 
$$
of extended deformations of the complex structure on $X$. 

The results of \cite{C2} show how one can construct a partition function for the TCFT associated to the Calabi-Yau elliptic algebra $\Hom_{\op{Com}(X) }(E,E)$.  This is a state in a Fock space associated to the periodic cyclic homology of the category $\op{Com}(X)$, with a certain symplectic pairing.  This symplectic vector space can be identified with $H^\ast(X)((t)) [d]$ where $d = \dim X$.  This partition function is an integral over the ``fundamental chain'' of the moduli spaces of surfaces. Conjecturally, this partition function is mirror to the generating function for Gromov-Witten invariants with descendants (these invariants are integrals over the fundamental class of Deligne-Mumford space).   Since the construction of this paper is completely explicit, it should be possible, in principle, to give a formula for this partition function.  The problem of explicitly writing down the partition function amounts, in our formalism, to finding a solution of the Sen-Zwiebach's quantum master equation in a certain dg Lie algebra construction from our explicit models $MG(n,m)$ for Segal's category of moduli spaces. 

In the A model, the Gromov-Witten invariants give operations on the compactified moduli spaces of surfaces, not just the uncompactifed spaces as in this paper.   A theorem of Katzarkov, Kontsevich, and Pantev \cite{K, KKP} shows that if  the circle action is homotopically trivial, then the action extends to the moduli space of stable curves.  The choice of a trivialisation of the circle action leads to an extension to the Deligne-Mumford compactification. The fact that the Hodge to de Rham spectral sequence degenerates for $X$ implies that this circle action is trivial. A choice of a splitting of the Hodge filtration on $X$ leads to a trivialisation of the circle action.

\bibliographystyle{amsalpha}

\end{document}